\documentclass[11pt,twoside]{article}

\usepackage{amsmath}
\usepackage{amssymb}
\usepackage{amsthm}

\allowdisplaybreaks

\def\ls{\lesssim}
\def\gs{\gtrsim}
\def\fz{\infty}

\renewcommand{\r}{\right}
\renewcommand{\l}{\left}

\def\supp{{\mathop\mathrm{\,supp\,}}}

\def\rr{{\mathbb R}}
\def\rn{{{\rr}^n}}
\def\zz{{\mathbb Z}}
\def\nn{{\mathbb N}}
\newcommand{\az}{\alpha}

\newcommand{\gz}{\gamma}

\newcommand{\lz}{\lambda}

\newcommand{\vz}{\varphi}

\newcommand{\wz}{\widetilde}
\newcommand{\oz}{\overline}

\newcommand{\cD}{{\mathcal D}}

\newcommand{\cF}{{\mathcal F}}

\newcommand{\cP}{{\mathcal P}}

\newcommand{\cS}{{\mathcal S}}

\newcommand{\db}{\dot{B}^s_{p,q}(\rn)}
\newcommand{\df}{\dot{F}^s_{p,q}(\rn)}
\newcommand{\dat}{\dot{A}_{p,q}^{s,\tau}(\rn)}
\newcommand{\dbt}{\dot{B}_{p,q}^{s,\tau}(\rn)}
\newcommand{\dft}{\dot{F}_{p,q}^{s,\tau}(\rn)}
\newcommand{\dsbt}{\dot{b}_{p,q}^{s,\tau}(\rn)}
\newcommand{\dsft}{\dot{f}_{p,q}^{s,\tau}(\rn)}
\newcommand{\dsat}{\dot{a}_{p,q}^{s,\tau}(\rn)}
\newcommand{\dah}{A\dot{H}_{p,q}^{s,\tau}(\rn)}
\newcommand{\dbh}{B\dot{H}_{p,q}^{s,\tau}(\rn)}
\newcommand{\dfh}{F\dot{H}_{p,q}^{s,\tau}(\rn)}
\newcommand{\dsbh}{b\dot{H}_{p,q}^{s,\tau}(\rn)}
\newcommand{\dsfh}{f\dot{H}_{p,q}^{s,\tau}(\rn)}
\newcommand{\dsah}{a\dot{H}_{p,q}^{s,\tau}(\rn)}

\def\gfz{\genfrac{}{}{0pt}{}}

\newcommand{\N}{{\mathbb N}}
\newcommand{\R}{{\mathbb R}}
\newcommand{\C}{{\mathbb C}}
\newcommand{\Z}{{\mathbb Z}}

\newtheorem{theorem}{Theorem}[section]
\newtheorem{lemma}{Lemma}[section]
\newtheorem{corollary}{Corollary}[section]
\newtheorem{proposition}{Proposition}[section]
\theoremstyle{definition}
\newtheorem{remark}{Remark}[section]
\newtheorem{definition}{Definition}[section]

\numberwithin{equation}{section}

\def\hs{\hspace{0.3cm}}

\begin{document}

\arraycolsep=1pt

\arraycolsep=1pt

\title{{\vspace{-5cm}\small\hfill\bf J. Math. Anal. Appl., to appear}\\
\vspace{4.5cm}\Large\bf Decompositions of
Besov-Hausdorff and Triebel-Lizorkin-Hausdorff Spaces and Their
Applications\footnotetext {\hspace{-0.35cm} 2000 {\it Mathematics
Subject Classification}. Primary 46E35; Secondary 42C40, 47G30.\endgraf
{\it Key words and phrases}. $\varphi$-transform, Hausdorff
capacity, Besov space, Triebel-Lizorkin space, embedding, atom,
molecule, trace, pseudo-differential operator.
\endgraf
The second author is supported by Grant-in-Aid for Young
Scientists (B) (No. 21740104) of
Japan Society for the Promotion of Science.
The third (corresponding) author is supported by the National
Natural Science Foundation (Grant No. 10871025) of China.}}
\author{Wen Yuan, Yoshihiro Sawano
and Dachun Yang\,\footnote{Corresponding author}}
\date{}
\maketitle

\begin{center}
\begin{minipage}{13cm}
{\small {\bf Abstract}\quad Let $p\in(1,\infty)$, $q\in[1,\infty)$,
$s\in\mathbb{R}$ and $\tau\in[0, 1-\frac{1}{\max\{p,q\}}]$. In this paper,
the authors establish the $\varphi$-transform characterizations of
Besov-Hausdorff spaces $B{\dot H}_{p,q}^{s,\tau}(\mathbb{R}^n)$ and
Triebel-Lizorkin-Hausdorff spaces $F{\dot
H}_{p,q}^{s,\tau}(\mathbb{R}^n)$ ($q>1$); as applications, the
authors then establish their embedding properties
(which on $B{\dot H}_{p,q}^{s,\tau}(\mathbb{R}^n)$ is also sharp),
smooth atomic and molecular decomposition
characterizations for suitable $\tau$.
Moreover, using their atomic and molecular decomposition
characterizations, the authors investigate the trace properties and
the boundedness of pseudo-differential operators with homogeneous
symbols in $B{\dot H}_{p,q}^{s,\tau}(\mathbb{R}^n)$ and $F{\dot
H}_{p,q}^{s,\tau}(\mathbb{R}^n)$ ($q>1$), which generalize the
corresponding classical results on homogeneous Besov and Triebel-Lizorkin spaces
when $p\in(1,\infty)$ and $q\in[1,\infty)$ by taking $\tau=0$.}
\end{minipage}
\end{center}

\vspace{0.1cm}

\section{Introduction}

\hskip\parindent To establish the connections between Besov and
Triebel-Lizorkin spaces with $Q$ spaces, which was an open problem
proposed by Dafni and Xiao in \cite{dx}, Yang and Yuan \cite{yy1,
yy2} introduced new classes of Besov-type spaces $\dbt$ and
Triebel-Lizorkin-type spaces $\dft$, which unify and generalize the
Besov spaces $\db$, Triebel-Lizorkin spaces $\df$, Morrey spaces,
Morrey-Triebel-Lizorkin spaces and $Q$ spaces. We pointed out that
the $Q$ spaces on $\rn$ were originally introduced by Ess\'en,
Janson, Peng and Xiao \cite{ejpx}; see also \cite{dx,ejpx,X1,X2} for
the history of $Q$ spaces and their properties.

Let $p\in(1,\infty)$, $q\in[1,\infty)$, $s\in\mathbb{R}$ and
$\tau\in[0, \frac{1}{(\max\{p,q\})'}]$, where and in what follows,
$t'$ denotes the conjugate index of $t\in [1,\fz)$. The
Besov-Hausdorff spaces $\dbh$ and Triebel-Lizorkin-Hausdorff spaces
$\dfh$ ($q>1$) were also introduced in \cite{yy1, yy2}; moreover, it
was proved therein that they are, respectively, the predual spaces
of $\dot{B}_{p',q'}^{-s,\tau}(\mathbb{R}^n)$ and
$\dot{F}_{p',q'}^{-s,\tau}(\mathbb{R}^n)$. The spaces $\dbh$ and
$\dfh$ were originally called the Hardy-Hausdorff spaces in
\cite{yy1, yy2}. However, it seems that it is more reasonable to
call them, respectively, the Besov-Hausdorff spaces and the
Triebel-Lizorkin-Hausdorff spaces. The spaces $\dbh$ and $\dfh$
unify and generalize the Besov space $\db$, the Triebel-Lizorkin
space $\df$ and the Hardy-Hausdorff space $HH^1_{-\alpha}(\rn)$,
where $HH^1_{-\alpha}(\rn)$ was introduced in \cite{dx} and was
proved to be the predual space of the space $Q_\alpha(\rn)$ therein.

It is well known that the wavelet decomposition plays an important
role in the study of function spaces and their applications;
see, for example, \cite{m92,mc97} and their references.
Moreover, the $\varphi$-transform
decomposition of Frazier and Jawerth \cite{fj85,fj88,fj}
is very similar in spirit to the wavelet decomposition, which
is also proved to be a powerful tool in the study
of function spaces and boundedness of operators, and
was further developed by Bownik \cite{b05, b07}.
In this paper, we establish the $\varphi$-transform
characterizations of the spaces $\dbh$ and $\dfh$; via
these characterizations, we also obtain their embedding properties
(which on $B{\dot H}_{p,q}^{s,\tau}(\mathbb{R}^n)$ is also sharp),
smooth atomic and molecular decomposition characterizations for
suitable $\tau$. Moreover, using their atomic and molecular
decomposition characterizations, we investigate the trace properties
and the boundedness of pseudo-differential operators with
homogeneous symbols (see \cite{gt}) in $\dbh$ and $\dfh$, which
generalizes the corresponding classical results on homogeneous Besov and
Triebel-Lizorkin spaces when $p\in (1,\fz)$ and $q\in[1,\fz)$ by
taking $\tau=0$; see, for example, Jawerth \cite[Theorem~5.1]{j77}
and \cite[Theorem~2.1]{j78} (or Frazier-Jawerth
\cite[Theorem~11.1]{fj}), and Grafakos-Torres
\cite[Theorems 1.1 and 1.2]{gt}. Recall that the
study of pseudo-differential operators with non-homogeneous symbols
on non-homogeneous Besov and Triebel-Lizorkin spaces
using $\vz$-transform arguments was started by Torres \cite{t90, t91};
the results in \cite{gt} are based on these works.
See also those articles for other references to previous work
on pseudo-differential operators on Triebel-Lizorkin spaces
using more classical methods. We will concentrate here
on $\vz$-transform arguments.

To recall the definitions of $\dbh$ and $\dfh$ in \cite{yy1, yy2},
we need some notation. Let $\cS(\rn)$ be the space of all Schwartz
functions on $\rn$. Following Triebel's \cite{t83}, set
$$\cS_\infty(\rn)\equiv\l\{\varphi\in\cS(\rn): \int_\rn
\varphi(x)x^\gamma\,dx=0\ \mbox{for all multi-indices}\ \gamma\in
\l(\nn\cup\{0\}\r)^n\r\}$$ and use $\cS'_\infty(\rn)$ to denote the
topological dual of $\cS_\infty(\rn)$, namely, the set of all
continuous linear functionals on $\cS_\infty(\rn)$ endowed with weak
$\ast$-topology. Recall that $\cS'(\rn)/\cP(\rn)$ and
$\cS'_\infty(\rn)$ are topologically equivalent, where
$\cS'(\rn)$ and $\cP(\rn)$ denote, respectively, the space
of all Schwartz distributions and the set of all polynomials
on $\rn$.

For each cube $Q$ in $\rn$, we denote its side length by
$\ell(Q)$, its center by $c_Q$, and set $j_Q\equiv-\log_2
\ell(Q)$. For $k=(k_1,\cdots,k_n)\in \zz^n$ and $j \in \zz$, let
$Q_{jk}$ be the dyadic cube $\{(x_1,\cdots, x_n):\ k_i \le
2^jx_i<k_i+1 \ \mathrm{for} \ i=1,\cdots,n\}\subset \rn$, $x_Q$
be the \textit{lower left-corner} $2^{-j}k$ of $Q=Q_{jk}$,
$\mathcal{D}(\rn)\equiv \{Q_{jk}\}_{j,\,k}$ and ${\mathcal
D}_j(\rn)\equiv\{ Q \in {\mathcal D}(\rn):\, \ell(Q)=2^{-j} \}$.
When dyadic cube $Q$ appears as an index, such as
$\sum_{Q\in\mathcal{D}(\rn)}$ and
$\{\cdot\}_{Q\in\mathcal{D}(\rn)}$, it is understood that $Q$
runs over all {\it dyadic cubes} in $\rn$.

For $x\in\rn$ and $r>0$, we write $B(x,\,r)\equiv \{y\in\rn:\,
|x-y|<r\}$. Next we recall the notion of Hausdorff capacities; see,
for example, \cite{ad, yy0}. Let $E \subset \rn$ and $d\in(0,\,n]$.
The {\it $d$-dimensional Hausdorff capacity of $E$} is defined by
\begin{equation*}
H^d(E)\equiv \inf \l\{\sum_jr_j^d:\,E\subset
\bigcup_jB(x_j,\,r_j)\r\},
\end{equation*} where the
infimum is taken over all covers $\{B(x_j,\,r_j)\}_{j=1}^\infty$ of
$E$ by countable families of open balls. It is well-known that $H^d$
is monotone, countably subadditive and vanishes on empty set.
Moreover, the notion of $H^d$ can be extended to $d=0$. In this
case, $H^0$ has the property that for all sets $E\subset\rn$,
$H^0(E)\ge1$, and $H^0(E)=1$ if and only if $E$ is bounded.

For any function $f: \rn \mapsto [0,\,\fz]$, the \textit{Choquet
integral of $f$ with respect to $H^d$} is defined by
$$\int_{\rn}f\,d H^d\equiv \int_0^{\fz}
H^d(\{x\in\rn:\ f(x)>\lz\})\,d\lz.$$ This functional is not
sublinear, so sometimes we need to use an equivalent integral with
respect to the $d$-dimensional dyadic Hausdorff capacity
$\widetilde{H}^d$, which is sublinear; see \cite{yy0} (also
\cite{yy1,yy2}) for the definition of dyadic Hausdorff capacities
and their properties.

Set $\R_+^{n+1}\equiv \rn\times (0,\fz)$. For any measurable
function $\omega$ on $\R_+^{n+1}$ and $x\in\rn$, we define its
\textit{nontangential maximal function} $N\omega (x)$ by setting
$N\omega(x)\equiv \sup_{|y-x|<t} |\,\omega(y,t)|.$

In what follows, for any $\varphi \in \cS(\R^n)$, we use $\cF\vz$ to
denote its Fourier transform, namely, for all $\xi\in\rn$,
$\cF\vz(\xi)=\int_\rn e^{-i\xi x}\vz(x)\,dx$. For all
$j\in\mathbb{Z}$ and $x\in\rn$, let $\varphi_j(x)\equiv
2^{jn}\varphi(2^jx)$. For any $p,\,q\in (0,\fz]$, let $(p \vee
q)\equiv \max\{p,\,q\}$; and for any $t\in [1,\fz]$, we denote
by $t'$ the conjugate index, namely, $1/t+1/t'=1$.

We now recall the notions of $\dbh$ and $\dfh$ in
\cite[Definition 5.2]{yy1} and \cite[Definition 6.1]{yy2}.

\begin{definition}\label{d1.1} Let $\varphi \in \cS(\R^n)$ such
that $\supp\cF\vz\subset\{\xi\in\rn:\ 1/2 \le |\xi| \le 2 \}$
and $\cF \varphi$ never vanishes
on $\{\xi\in\rn:\ 3/5 \le |\xi| \le 5/3 \}$.
Let $p\in(1,\fz)$ and $s \in \R$.

(i) If $q\in[1,\fz)$ and $\tau\in[0, \frac{1}{(p \vee q)'}]$,
{\it the Besov-Hausdorff space
$B\dot{H}_{p,q}^{s,\tau}(\rn)$} is then defined to be the set of all
$f \in \cS_\infty'(\R^n)$ such that
$$
\| \, f \, \|_{B\dot{H}_{p,q}^{s,\tau}(\R^n)} \equiv \inf_\omega
\left\{ \sum_{j \in \Z}2^{jsq} \l\| \varphi_j\ast f
[\omega(\cdot,2^{-j})]^{-1} \r\|^q_{L^p(\rn)}
\right\}^\frac1q<\fz,
$$
where $\omega$ runs over all nonnegative Borel measurable functions
on $\R_+^{n+1}$ such that
\begin{equation}
\label{1.1} \int_{\R^n} [N\omega(x)]^{(p \vee q)'} \,dH^{n\tau(p\vee
q)'}(x) \le 1
\end{equation}
and with the restriction that for any $j\in\zz$, $\omega(\cdot,
2^{-j})$ is allowed to vanish only where $\vz_j\ast f$ vanishes.

(ii) If $q\in(1,\fz)$ and $\tau\in[0,\frac{1}{(p \vee q)'}]$,
{\it the Triebel-Lizorkin-Hausdorff
space $F\dot{H}_{p,q}^{s,\tau}(\rn)$} is then
defined to be the set of
all $f \in \cS_\infty'(\R^n)$ such that
$$
\| \, f \, \|_{F\dot{H}_{p,q}^{s,\tau}(\rn)}
\equiv
\inf_\omega
\left\|
\left\{
\sum_{j \in \Z}2^{jsq}
\l|\varphi_j\ast f[\omega(\cdot,2^{-j})]^{-1}\r|^q
\right\}^\frac1q\right\|_{L^p(\rn)}<\fz,
$$
where $\omega$ runs over all nonnegative Borel measurable functions
on $\R_+^{n+1}$ such that $\omega$ satisfies \eqref{1.1} and with
the restriction that for any $j\in\zz$, $\omega(\cdot, 2^{-j})$ is
allowed to vanish only where $\vz_j\ast f$ vanishes.
\end{definition}

To simplify the presentation, in what follows, we use $\dah$ to
denote either $\dbh$ or $\dfh$. When $\dah$ denotes $\dfh$, then it
will be understood tacitly that $q\in(1,\fz)$. It was proved
in \cite[Proposition 5.1]{yy1} and \cite[Section 6]{yy2}
that the space $\dah$ is independent of the choices of $\varphi$.
We also remark that when
$\tau=0$, then $B{\dot H}_{p,q}^{s,0}(\rn)\equiv {\dot
B}_{p,q}^s(\rn)$ and $F{\dot H}_{p,q}^{s,0}(\rn)\equiv {\dot
F}_{p,q}^s(\rn)$; when $\alpha\in(0,1)$, $s=-\alpha$, $p=q=2$ and
$\tau=1/2-\alpha/n$, then $A{\dot
H}_{2,2}^{-\alpha,1/2-\alpha}(\rn)\equiv HH^1_{-\alpha}(\rn)$,
which is the predual space of $Q_\alpha(\rn)$.

We now recall the notions of Besov-type spaces $\dbt$ and
Triebel-Lizorkin-type spaces $\dft$ in \cite[Definition 1.1]{yy2}
and \cite[Definition 3.2]{yy1}.

\begin{definition}\label{d1.2}
Let $s\in\rr$, $\tau\in[0,\fz)$, $q\in(0,\,\fz]$ and $\varphi$ be as
in Definition \ref{d1.1}.

(i) If $p\in(0,\,\fz]$, {\it the Besov-type space $\dbt$} is defined
to be the set of all $f\in\cS_\fz'(\rn)$ such that
$\|f\|_{\dbt}<\fz$, where
\begin{equation*}
\|f\|_{\dbt}\equiv\sup_{P\in \mathcal{D}(\rn)}\dfrac{1}
{|P|^{\tau}}\l\{\sum^{\fz}_{j=j_P}\l[\int_P (2^{js}|\varphi_j \ast
f(x)|)^p\, dx\r]^{q/p}\r\}^{1/q}
\end{equation*}
with suitable modifications made when $p=\fz$ or $q=\fz$.

(ii) If $p\in(0,\,\fz)$, {\it the Triebel-Lizorkin-type space
$\dft$} is defined to be the set of all $f\in\cS_\fz'(\rn)$ such
that $\|f\|_{\dft}<\fz$, where
\begin{equation*}
\|f\|_{\dft}\equiv\sup_{P\in \mathcal{D}(\rn)}\dfrac{1}
{|P|^{\tau}}\l\{\int_P \l[\sum^{\fz}_{j=j_P}(2^{js}|\varphi_j \ast
f(x)|)^q\, dx\r]^{p/q}\r\}^{1/p}
\end{equation*}
with suitable modifications made when $q=\fz$.
\end{definition}

Similarly, we use $\dat$ to denote $\dbt$ or $\dft$. If $\dat$ means
$\dft$, then the case $p=\fz$ is excluded. It was proved in
\cite[Corollary 3.1]{yy2} that the space $\dat$ is independent of
the choices of $\varphi$. Also, \cite[Theorem 5.1]{yy1} and
\cite[Theorem 6.1]{yy2} show that
$(\dah)^*=\dot{A}_{p',q'}^{-s,\tau}(\rn)$ for all $s\in\R$,
$p\in(1,\fz)$, $q\in[1,\fz)$ and $\tau\in[0,\frac1{(p\vee q)'}]$.
This result partially extends the well-known dual results on Besov
spaces, Triebel-Lizorkin spaces and the recent result that
$(HH^1_{-\alpha}(\rn))^*=Q_{\alpha}(\rn)$ obtained in \cite[Theoren
7.1]{dx}.

We remark that when $\tau=0$, then ${\dot B}_{p,q}^{s,0}(\rn)\equiv
{\dot B}_{p,q}^s(\rn)$ and ${\dot F}_{p,q}^{s,0}(\rn)\equiv {\dot
F}_{p,q}^s(\rn)$; when $\alpha\in(0,1)$, $s=\alpha$, $p=q=2$ and
$\tau=1/2-\alpha/n$, then ${\dot
A}_{2,2}^{\alpha,1/2-\alpha}(\rn)\equiv Q_\alpha(\rn)$; see
\cite[Corollary 3.1]{yy1}. It was proved in \cite{syy1} that
Besov-Morrey spaces in \cite{st} are proper subspaces of $\dbt$ and
that Triebel-Lizorkin-Morrey spaces in \cite{st} are special cases
of $\dft$. It was also proved in \cite{st} that Morrey spaces are
special cases of Triebel-Lizorkin-Morrey spaces. The $\vz$-transform
characterizations, embedding properties, smooth atomic and molecular
decomposition characterizations of $\dat$ were obtained in
\cite{yy2}, which were further applied in \cite{syy1} to establish
their trace properties and the boundedness of pseudo-differential
operators with homogeneous symbols in $\dat$.

In Section 2 of this paper, we establish the $\varphi$-transform
characterizations (see Theorem \ref{t2.1} below) and embedding
properties (Proposition \ref{p2.2} below) of $\dah$. In particular,
we show, in Proposition \ref{p2.3} below, that the embedding
property of $\dbh$ is sharp. Using these $\varphi$-transform
characterizations, in Section 3 below, we obtain the boundedness of
almost diagonal operators and the smooth atomic and molecular
decomposition characterizations of $\dah$. As applications of these
decomposition characterizations, in Section 4 of this paper, we
investigate the trace properties (see Theorem \ref{t4.2} below) and
the boundedness of pseudo-differential operators with homogeneous
symbols in $\dah$ (see Theorem \ref{t4.1} below). We pointed out that
the method used in the proof of Theorem \ref{t4.1} comes
from \cite{ftw,fhjw,t90,t91,gt}.

Notice that the spaces $\dah$ are only known to be quasi-normed
spaces so far due to the infimum on $\omega$ appearing in their definitions,
which satisfies the condition \eqref{1.1}.
This brings us some essential difficulties, comparing with the methods used in
\cite{yy2,syy1} for the spaces $\dbt$ and $\dft$.
To overcome these new difficulties, we use the Aoki theorem (see \cite{ao} and
the proof of Theorem \ref{t3.1} below) and establish some subtly equivalent
characterizations on the Hausdorff capacity (see Lemmas \ref{l2.4}, \ref{l3.1}
and \ref{l4.1} below). These characterizations on the Hausdorff capacity
are geometrical, whose proofs are constructive and
invoke some covering lemmas. Propositions \ref{p2.2} and
\ref{p2.3} and Theorem \ref{t3.1} below reflect the differences between
the spaces $\dbh$ and $\dfh$ and the spaces $\dbt$ and $\dft$;
see also Remarks \ref{r2.3} and \ref{r3.1} below.

Finally, we make some conventions on notation. Throughout the whole
paper, we denote by $C$ a positive constant which is independent of
the main parameters, but it may vary from line to line. The symbol
$A\ls B$ means that $A\le CB$. If $A\ls B$ and $B\ls A$, then we
write $A\sim B$. If $E$ is a subset of $\rn$, we denote by $\chi_E$
the characteristic function of $E$. For all $Q \in \cD(\rn)$ and
$\vz\in\cS(\rn)$, set
$\varphi_Q(x)\equiv|Q|^{-1/2}\varphi(2^{j_Q}(x-x_Q))$ and
$\widetilde{\chi}_Q(x)\equiv|Q|^{-1/2}\chi_Q(x)$ for all $x\in\rn$.
We also set $\nn\equiv\{1,\, 2,\, \cdots\}$ and
$\zz_+\equiv(\nn\cup\{0\})$.

\section{The $\varphi$-transform characterizations}

\hskip\parindent In this section, we establish the
$\varphi$-transform characterizations of the spaces $\dah$ in the
sense of Frazier and Jawerth; see, for example, \cite{fj85,fj88,fj,fjw}. We
begin with the definition of the corresponding sequence space of
$\dah$.

\begin{definition}\label{d2.1} Let $p\in(1,\fz)$ and $s\in\rr$.

(i) If $q\in[1,\fz)$ and $\tau\in[0, \frac{1}{(p \vee q)'}]$, {\it
the sequence space $\dsbh$} is then defined to be the set of all
$t=\{t_Q\}_{Q \in \cD(\rn)}\subset \mathbb{C}$ such that
\begin{equation*}
\| \, t \, \|_{\dsbh} \equiv \inf_\omega \left\{ \sum_{j \in
\Z}2^{jsq} \left\| \sum_{Q \in \cD_j(\rn)}|t_Q|\widetilde{\chi}_Q
[\omega(\cdot,2^{-j})]^{-1} \right\|^q_{L^p(\rn)}
\right\}^\frac1q<\fz,
\end{equation*}
where the infimum is taken over all nonnegative Borel measurable
functions $\omega$ on $\rr^{n+1}_+$ such that $\omega$ satisfies
\eqref{1.1} and with the restriction that for any $j\in\zz$,
$\omega(\cdot,2^{-j})$ is allowed to vanish only where $\sum_{Q \in
\cD_j(\rn)}|t_Q|\widetilde{\chi}_Q$ vanishes.

(ii) If $q\in(1,\fz)$ and $\tau\in[0, \frac{1}{(p \vee q)'}]$, {\it
the sequence space $\dsfh$} is then defined to be the set of all
$t=\{t_Q\}_{Q \in \cD(\rn)}\subset \mathbb{C}$ such that
\begin{equation*}
\| \, t \, \|_{\dsfh} \equiv \inf_\omega \left\| \left\{ \sum_{j \in
\Z}2^{jsq} \left(\sum_{Q \in \cD_j(\rn)}|t_Q|\widetilde{\chi}_Q
[\omega(\cdot,2^{-j})]^{-1}\right)^q
\right\}^\frac1q\right\|_{L^p(\rn)}<\fz,
\end{equation*}
where the infimum is taken over all nonnegative Borel measurable
functions $\omega$ on $\rr^{n+1}_+$ with the same restrictions as in
(i).
\end{definition}

Similarly, in what follows, we use $\dsah$ to denote either $\dsbh$
or $\dsfh$. When $\dsah$ denotes $\dsfh$, then it will be understood
tacitly that $q\in(1,\fz)$. We remark that $\|\cdot\|_{\dsah}$ is a
quasi-norm, namely, there exists a nonnegative constant
$\rho\in[0,1]$ such that for all $t_1,\, t_2\in \dsah$,
\begin{equation}\label{2.1}
\|t_1+t_2\|_{\dsah}\le 2^\rho (\|t_1\|_{\dsah}+\|t_2\|_{\dsah}).
\end{equation}

\begin{remark}\label{r2.1}
On \eqref{1.1}, we observe that if $0<a\le b\le \frac1{\tau}$, then
for all nonnegative measurable functions $\omega$ on $\R^{n+1}_+$,
$\int_{\rn}\l[N\omega(x)\r]^{a}\, d H^{n\tau a}(x)<\fz$ induces
$\int_{\rn}\l[N\omega(x)\r]^{b}\, d H^{n\tau b}(x)<\fz.$ In fact,
without loss of generality, we may assume that
$\int_{\rn}[N\omega(x)]^{a}\, d H^{n\tau a}(x)\le1$. For all
$l\in\zz$, set $E_l\equiv\{x\in\rn:\ N\omega(x)>2^l\}$. Then
\begin{eqnarray*}
1\ge \int_{\rn}\l[N\omega(x)\r]^{a}\, d H^{n\tau a}(x) \sim
\sum_{l\in\zz}2^{la} H^{n\tau a}(E_l).
\end{eqnarray*}
For each $l\in\zz$, we choose a ball covering $\{B(x_{jl},
r_{jl})\}_j$ of $E_l$ that almost attains $H^{n\tau a}(E_l)$\,:\,
$H^{n\tau a}(E_l)\sim \sum_j
r_{jl}^{n\tau a}$. Thus, $\sum_{l\in\zz}2^{la}\sum_j r_{jl}^{n\tau
a}\ls1$, and hence, for all $j$ and $l$, $2^lr_{jl}^{n\tau}\ls1$.
Then $2^{lb}r_{il}^{n\tau b}\ls 2^{la}r_{il}^{n\tau a}$ since $a\le
b$ and
\begin{eqnarray*}
\int_{\rn}\l[N\omega(x)\r]^{b}\, d H^{n\tau b}(x) &&\sim
\sum_{l\in\zz}2^{lb} H^{n\tau b}(E_l)\ls \sum_{l\in\zz}2^{lb}\sum_j
r_{jl}^{n\tau b}\ls\sum_{l\in\zz}2^{la}\sum_j r_{jl}^{n\tau a},
\end{eqnarray*}
which yields the above claim.
\end{remark}

Let $\varphi$ be as in Definition \ref{d1.1}. For all $x\in\rn$, set
$\wz{\vz}(x)\equiv \oz{\vz(x)}$. Then by \cite[Lemma (6.9)]{fjw},
there exists a function $\psi\in\cS(\rn)$ such that
$\supp\cF\psi\subset \{\xi\in\rn:\ 1/2 \le |\xi| \le 2 \}$, $\cF
\psi$ never vanishes on $\{\xi\in\rn:\ 3/5 \le |\xi| \le 5/3 \}$ and
that for all $\xi\in\rn$, $\sum_{j\in\mathbb{Z}} \cF{\widetilde{\varphi}}
(2^{-j}\xi)\cF\psi(2^{-j}\xi)=\chi_{\rn\setminus\{0\}}(\xi)$.
Furthermore, we have the following \textit{Calder\'on reproducing
formula} which asserts that for all $f\in\cS'_\fz(\rn)$,
\begin{equation}\label{2.2}
f=\sum_{j\in\mathbb{Z}}\psi_j\ast \widetilde{\varphi}_j\ast
f=\sum_{Q\in\cD(\rn)} \langle f,\varphi_Q \rangle \psi_Q
\end{equation}
in $\cS'_\fz(\rn)$; see \cite[Lemma 2.1]{yy2}.

Now we recall the notion of the $\varphi$-transform; see, for
example, \cite{fj85,fj88,fj,fjw}.

\begin{definition}\label{d2.2}
Let $\varphi$, $\psi\in\cS(\rn)$ such that $\supp\cF\vz$,
$\supp\cF\psi\subset\{\xi\in\rn:\ 1/2 \le |\xi| \le 2 \}$, $\cF
\varphi$, $\cF \psi$ never vanish on $\{\xi\in\rn:\ 3/5 \le |\xi|
\le 5/3 \}$ and $\sum_{j \in
\Z}\cF(\widetilde{\varphi}_j)\cF(\psi_j) \equiv \chi_{\rn \setminus
\{0\}}.$

(i) The {\it $\varphi$-transform $S_\varphi$} is defined to be
the map taking each $f\in\cS'_\fz(\rn)$ to the sequence $
S_\varphi f \equiv \{(S_\varphi f)_Q\}_{Q \in \cD(\rn)}$, where
$(S_\varphi f)_Q\equiv\langle f,\varphi_Q \rangle$ for all
$Q\in\cD(\rn)$.

(ii) The {\it inverse $\varphi$-transform $T_\psi$} is defined to
be the map taking a sequence
$t=\{t_Q\}_{Q\in\cD(\rn)}\subset\mathbb{C}$ to $T_\psi t \equiv
\sum_{Q \in \cD(\rn)}t_Q \psi_Q$.
\end{definition}

To show that $T_\psi$ is well defined for all $t\in\dsah$, we
need the following conclusion.

\begin{lemma}\label{l2.1}
Let $p\in(1,\fz), \, q\in[1,\fz), \, s \in \R$ and $\tau\in[0,
\frac{1}{(p \vee q)'}]$. Then for all $t\in\dsah$, $T_\psi
t=\sum_{Q \in \cD(\rn)}t_Q \psi_Q$ converges in
$\cS'_\infty(\rn)$; moreover, $T_\psi: \dsah \to
\cS'_\infty(\rn)$ is continuous.
\end{lemma}

\begin{proof}
By similarity, we only consider the space $\dsbh$.

Let $t=\{t_Q\}_{Q\in\cD(\rn)}\in\dsbh$. We need to show that there
exists an $M\in\zz_+$ such that for all $\phi\in\cS_\fz(\rn)$,
$\sum_{Q \in \cD(\rn)}|t_Q||\langle \psi_Q,\phi \rangle| \lesssim \|
\phi \|_{\cS_M}$, where and in what follows, for all $M\in\zz_+$ and
$\vz\in\cS(\rn)$, we set $\|\vz\|_{\cS_M}\equiv \sup_{|\gamma|\le M
}\sup_{x\in \rn} |\partial^\gamma\vz (x)|(1+|x|)^{n+M+|\gamma|}.$

Choose a Borel function $\omega$ that almost attains the infimum in
Definition \ref{d2.1} (i). That is, $\omega$ is a function on $\rr^{n+1}_+$
satisfying \eqref{1.1} as well as
\begin{equation}\label{2.3}
\left\{ \sum_{j \in \Z}2^{jsq} \left\| \sum_{Q \in
\cD_j(\rn)}|t_Q|\widetilde{\chi}_Q [\omega(\cdot,2^{-j})]^{-1}
\right\|^q_{L^p(\rn)} \right\}^\frac1q \le 2 \| \, t \,
\|_{\dsbh}.
\end{equation}
A simple consequence obtained from (\ref{1.1}) is that for all
$(x,s) \in \R^{n+1}_+$, $\omega(x,s) \lesssim s^{-n\tau}$; see
\cite[Remark 4.1]{yy1}. Then for all $Q \in \cD_j(\rn)$, by
H\"older's inequality and \eqref{2.3}, we have
\begin{align}
\label{2.4}
|t_Q|\le |Q|^{-\tau-\frac1p}|t_Q| \left(\int_Q
[\omega(x,2^{-j})]^{-p}\,dx\right)^\frac{1}{p} \lesssim
|Q|^{\frac{s}{n}+\frac{1}{2}-\tau-\frac{1}{p}} \|t \|_{\dsbh}.
\end{align}
Recall that as a special case of \cite[Lemma 2.11]{b07}, there exists a positive
constant $L_0$ such that for all $j\in\zz$,
\begin{equation}\label{2.5}
\sum_{Q\in\cD_j(\rn)}\l(1+\frac{|x_Q|^n}{\max\{1,|Q|\}}\r)^{-L_0}\ls 2^{n|j|}.
\end{equation}
Furthermore, it was proved in \cite[p.\,10]{yy2} that if $
L>\max\{1/p+1/2-s/n-\tau,1/p+3/2+s/n+\tau, L_0\}$,
then there exists an $M\in\zz_+$ such that
for all $Q\in\cD_j(\rn)$,
\begin{equation} \label{2.6} |\langle\psi_Q,
\phi\rangle|\ls \|\phi\|_{\cS_M}\l(1+
\frac{|x_Q|^n}{\max\{1,|Q|\}}\r)^{-L}\l(\min\{2^{-jn},
2^{jn}\}\r)^L;
\end{equation}
see also \cite[(3.18)]{b07}. Using \eqref{2.4}, \eqref{2.6} and
\eqref{2.5}, we conclude that
\begin{eqnarray*}
\sum_{Q \in \cD(\rn)}|t_Q| |\langle \psi_Q,\phi \rangle|
&&\lesssim \| t \|_{\dsbh}\| \phi \|_{\cS_M}\\
&&\hs\times\sum_{Q \in \cD(\rn)}
|Q|^{\frac{s}{n}+\frac{1}{2}-\tau-\frac{1}{p}}
\left(1+\frac{|x_Q|^n}{\max\{1,|Q|\}}\right)^{-L}
2^{-L|j_Q|n}\\
&&\lesssim \|t \|_{\dsbh} \| \phi\|_{\cS_M},
\end{eqnarray*}
which completes the proof of Lemma \ref{l2.1}.
\end{proof}

Now we are ready to present our main result of this section.

\begin{theorem}\label{t2.1} Let $p\in(1,\fz), \,
q\in[1,\fz), \, s \in \R$, $\tau\in[0, \frac{1}{(p \vee q)'}]$,
$\varphi$ and $\psi$ be as in Definition \ref{d2.2}. Then
$S_\varphi: \dah \to \dsah$ and $T_\psi: \dsah \to \dah$ are
bounded; moreover, $T_\psi\circ S_\varphi$ is the identity on
$\dah$.
\end{theorem}

To prove Theorem \ref{t2.1}, we need some technical lemmas. For a
sequence $t=\{t_Q\}_{Q \in \cD(\rn)}$, $Q \in \cD(\rn)$,
$r\in(0,\fz]$ and
$\lambda\in(0,\fz)$, define
\[
(t^*_{r,\lambda})_Q \equiv \left( \sum_{P \in \cD_{j_Q}(\rn)}
\frac{|t_P|^r}{(1+[\ell(P)]^{-1}|x_P-x_Q|)^\lambda} \right)^\frac1r
\]
and $t^*_{r,\lambda}\equiv\{(t^*_{r,\lambda})_Q\}_{Q \in \cD(\rn)}$.
For any $p,q\in(0,\fz]$, let $p\wedge q\equiv \min\{p,q\}$. The
following estimate is crucial in that this corresponds to the
maximal operator estimate.

\begin{lemma}\label{l2.2} Let $s,\,p,\,q,\,\tau$ be as in
Theorem \ref{t2.1} and $\lambda \in (n,\fz)$ be sufficiently large.
Then there exists a positive constant $C$ such that for all
$t\in\dsah$, $\| t \|_{\dsah}\le \| t^*_{p \wedge q,\lambda}
\|_{\dsah}\le C\| t \|_{\dsah}$.
\end{lemma}

\begin{proof}
The inequality $\| t \|_{\dsah}\le \| t^*_{p \wedge
q,\lambda}\|_{\dsah}$ being trivial, we only need to concentrate on $\|
t^*_{p \wedge q,\lambda}\|_{\dsah}\lesssim \| t \|_{\dsah}$.
Also, by similarity, we only consider the spaces $\dsbh$.

Let $t=\{t_Q\}_{Q \in \cD(\rn)} \in \dsbh$. We choose a Borel
function $\omega$ as in the proof of Lemma \ref{l2.1}. For all cubes
$Q \in \cD_j(\rn)$ and $m\in\nn$, we set $A_0(Q)\equiv \{P \in
\cD_j(\rn) \, : \, 2^j|x_P-x_Q| \le 1 \}$ and $ A_m(Q)\equiv \{P \in
\cD_j(\rn) \, : \, 2^{m-1} < 2^j|x_P-x_Q| \le 2^m \}. $ The triangle
inequality that $|x-y| \le |x-x_Q|+|x_Q-x_P|+|x_P-y|$ gives us that
$|x-y| \le 3\sqrt{n}2^{m-j}$ provided $x \in Q$, $y \in P$ and $P
\in A_m(Q)$.

For all $m \in \zz_+$ and $(x,s)\in\R^{n+1}_+$, we set
\[
\omega_m(x,s)\equiv2^{-mn(\lfloor(p \vee q)' \rfloor+2)} \sup\{
\omega(y,s) \, : \, y\in\rn,\ |y-x|<\sqrt{n}2^{m+2}s \},
\]
where and in what follows, $\lfloor s\rfloor$ denotes the
\textit{maximal integer no more than $s$}. By the argument in
\cite[Lemma 5.2]{yy1}, we know that $\omega_m$ still satisfies
(\ref{1.1}) modulo multiplicative constants independent of $m$. Also
it follows from the definition of $\omega_m$ that for all $x\in Q$,
$y\in P$ with $P\in A_m(Q)$, $\omega(y,2^{-j}) \le 2^{mn(\lfloor(p
\vee q)' \rfloor+2)} \omega_m(x,2^{-j}).$ For all $r\in(0,\fz)$ and
$a\in (0,r)$, using this estimate and the monotonicity
of $l^{a/r}$, we obtain that for all $x \in Q$,
\begin{align*}
\lefteqn{ \sum_{P \in A_m(Q)}
\frac{|t_P|^r}{(1+2^j|x_Q-x_P|)^\lambda}
[\omega_m(x,2^{-j})]^{-r}}\\
&\le\left\{\sum_{P \in A_m(Q)}
\frac{|t_P|^a}{(1+2^j|x_Q-x_P|)^{\lambda a/r}}
[\omega_m(x,2^{-j})]^{-a}\right\}^{r/a}\\
&\lesssim
2^{-m\lambda+jnr/a}\left\{
\int_{\R^n}\sum_{P \in A_m(Q)}
|t_P|^a\chi_P(y)[\omega_m(x,2^{-j})]^{-a}\,dy\right\}^{r/a}\\
&\lesssim2^{-m\lambda+nr\{j/a+m(\lfloor(p \vee q)' \rfloor+2)\}}
\left\{\int_{\R^n}\sum_{P \in A_m(Q)}
|t_P|^a\chi_P(y)[\omega(y,2^{-j})]^{-a}\,dy\right\}^{r/a}\\
&\lesssim 2^{-m\lambda+mnr(1/a+\lfloor(p \vee q)' \rfloor+2)}
\left\{\mathrm{HL} \left( \sum_{P \in
A_m(Q)}|t_P|^a\chi_P[\omega(\cdot,2^{-j})]^{-a} \right)(x)\right\}^{r/a},
\end{align*}
where $\mathrm{HL}$ denotes the Hardy-Littlewood maximal operator on
$\rn$.

For all $m\in\zz_+$, set $ t^{*,m}_{r,\lambda}\equiv
\{(t^{*,m}_{r,\lambda})_Q\}_{Q \in \cD(\rn)}$ with
$$(t^{*,m}_{r,\lambda})_Q \equiv \left( \sum_{P \in A_m(Q)}
\frac{|t_P|^r}{(1+[\ell(P)]^{-1}|x_P-x_Q|)^\lambda} \right)^\frac1r.$$
In what follows, choose $a\in (0,p\wedge q)$ and
$\lambda>(p\wedge q)[n(1/a+\lfloor (p \vee q)' \rfloor+2)+\rho]$,
where $\rho$ is a nonnegative constant
as in \eqref{2.1}. By \eqref{2.1}, the previous pointwise estimate
and the $L^{\frac pa}(\rn)$-boundedness of $\mathrm{HL}$,
we obtain
\begin{eqnarray*}
&&\| t^*_{p \wedge q,\lambda} \|_{\dsbh}\\
&&\hs\lesssim
\sum_{m=0}^\infty 2^{\rho m} \| t^{*,m}_{p \wedge q,\lambda}
\|_{\dsbh}\\
&&\hs\lesssim \sum_{m=0}^\infty 2^{\rho m} \left\{ \sum_{j \in \Z}
2^{jsq} \left[ \int_{\R^n}\sum_{Q \in \cD_j(\rn)}
\left( \sum_{P \in A_m(Q)} \frac{|t_P|^{p
\wedge q}}{(1+[\ell(P)]^{-1}|x_P-x_Q|)^\lambda} \right)^\frac{p}{p
\wedge q}\r.\r. \\
&&\hs\hs\times\l.\l.
\frac{\widetilde{\chi}_Q(x)^p}{[\omega_m(x,2^{-j})]^p}\,dx
\right]^\frac{q}{p}
\right\}^\frac{1}{q}\\
&&\hs\lesssim \sum_{m=0}^\infty 2^{-\frac{m}{p\wedge
q}\{\lambda-(p\wedge q)[n(1/a+\lfloor (p \vee q)' \rfloor+2)+\rho]\}}\\
&&\hs\hs\times \left[ \sum_{j \in \Z} 2^{jsq} \left\{ \int_{\R^n}
\l[\mathrm{HL} \left( \sum_{P \in \cD_j(\rn)}
\frac{(|t_P|\widetilde{\chi}_P)^a}
{[\omega(\cdot,2^{-j})]^a} \right)(x)\r]^\frac{p}a
\,dx \right\}^\frac{q}{p} \right]^\frac{1}{q}
\lesssim\|t\|_{\dsbh},
\end{eqnarray*}
which completes the proof of Lemma \ref{l2.2}.
\end{proof}

For any $f\in\cS'_\fz(\rn)$, $\gamma\in\zz_+$ and
$Q\in\mathcal{D}_j(\rn)$, set $\sup_Q(f)\equiv|Q|^{1/2}\sup_{y\in
Q}|\widetilde\varphi_j\ast f(y)|$ and $${\rm
inf}_{Q,\,\gamma}(f)\equiv|Q|^{1/2}\max\l\{\inf_{y\in \widetilde
{Q}}|\widetilde\varphi_j\ast f(y)|:\ \ell(\widetilde
Q)=2^{-\gamma}\ell(Q),\,\widetilde{Q}\subset Q\r\}.$$ Let
$\sup(f)\equiv\{\sup_Q(f)\}_{Q\in\cD(\rn)}$ and
$\inf_\gamma(f)\equiv\{\inf_{Q,\,\gamma}(f)\}_{Q\in\cD(\rn)}$. We
have the following conclusion, whose proof is similar to \cite[Lemma
2.5]{fj} and we omit the details.

\begin{lemma}\label{l2.3}
Let $s,\,p,\,q,\,\tau$ be as in Theorem \ref{t2.1} and
$\gamma\in\zz_+$ be sufficiently large. Then there exists a constant
$C\in[1,\,\fz)$ such that for all $f\in\dah$,
$$C^{-1}\|{\rm
inf}_{\gamma}(f)\|_{\dsah}\le\|f\|_{\dah}\le\|\sup(f)\|_{\dsah}\le
C\|{\rm inf}_{\gamma}(f)\|_{\dsah}.$$
\end{lemma}

With the Calder\'on reproducing formula \eqref{2.2}, Lemmas
\ref{l2.2} and \ref{l2.3}, the proof of Theorem \ref{t2.1} follows
the method pioneered by Frazier and Jawerth (see
\cite[pp.\,50-51]{fj}); see also the proof of \cite[Theorem
3.5]{bh}. We omit the details.

Recall that the corresponding sequence spaces $\dsat$ of $\dat$ in
\cite[Definition 3.1]{yy2} were defined as follows.

\begin{definition}\label{d2.3}
Let $s\in\rr$, $q\in(0,\fz]$ and $\tau\in(0,\fz)$. {\it The
sequence space $\dsat$} is defined to be the set of all
$t=\{t_Q\}_{Q\in\cD(\rn)}\subset\mathbb{C}$ such that
$\|t\|_{\dsat}<\fz$, where if $\dsat\equiv\dsbt$ for $p\in (0,\fz]$, then
\begin{equation*}
\|t\|_{\dsbt}\equiv\sup_{P\in\cD(\rn)}\dfrac{1}
{|P|^{\tau}}\l\{\sum^{\fz}_{j=j_P}2^{jsq}\l[\int_P
\l(\sum_{l(Q)=2^{-j}}|t_Q|\widetilde{\chi}_Q(x)\r)^p\,
dx\r]^{q/p}\r\}^{1/q}
\end{equation*}
and if $\dsat\equiv\dsft$ for $p\in(0,\fz)$, then
\begin{equation*}
\|t\|_{\dsft}\equiv\sup_{P\in\cD(\rn)}\dfrac{1}
{|P|^{\tau}}\l\{\int_P\l[\sum_{Q\subset
P}\l(|Q|^{-s/n}|t_Q|\widetilde{\chi}_Q(x)\r)^q\r]^{p/q}\,dx\r\}^{1/p}.
\end{equation*}
\end{definition}

We now establish the duality between $\dsah$ and
$\dot{a}_{p',q'}^{-s,\tau}(\rn)$, which is used in
Sections 3 and 4 below. In what follows, for any quasi-Banach spaces
${\mathcal B}_1$ and ${\mathcal B}_2$, the symbol ${\mathcal
B}_1\hookrightarrow{\mathcal B}_1$ means that there exists a
positive constant $C$ such that for all $f\in{\mathcal B}_1$, then
$f\in {\mathcal B}_2$ and $\|f\|_{{\mathcal B}_2}\le
C\|f\|_{{\mathcal B}_1}$.

\begin{proposition}\label{p2.1}
Let $s,\,p,\,q,\,\tau$ be as in Theorem \ref{t2.1}. Then $(\dsah)^*=
\dot{a}_{p',q'}^{-s,\tau}(\rn)$ in the following sense.

If $t=\{t_Q\}_{Q\in\cD(\rn)}\in \dot{a}_{p',q'}^{-s,\tau}(\rn)$, then
the map
$$\lambda=\{\lambda_Q\}_{Q\in\cD(\rn)}\mapsto \langle
\lambda, t\rangle\equiv \sum_{Q\in\cD(\rn)}\lambda_Q\overline{t}_Q$$
defines a continuous linear functional on $\dsah$ with operator norm
no more than a constant multiple of $\|t\|_{\dot{a}_{p',q'}^{-s,\tau}(\rn)}$.

Conversely, every $L\in (\dsah)^*$ is of this form for a certain
$t\in \dot{a}_{p',q'}^{-s,\tau}(\rn)$ and
$\|t\|_{\dot{a}_{p',q'}^{-s,\tau}(\rn)}$ is no more than a constant
multiple of the operator norm of $L$.
\end{proposition}

\begin{proof}
We only consider the spaces $\dsbh$ because the assertion for
$\dsfh$ can be proved similarly. Below we write $\rr^{n+1}_\zz\equiv
\{(x,a) \in \rr^{n+1}_+ \, : \, \log_2 a \in \Z \}$.

For $t=\{t_Q\}_{Q\in\cD(\rn)}\in \dot{b}_{p',q'}^{-s,\tau}(\rn)$ and
$\lambda=\{\lambda_Q\}_{Q\in\cD(\rn)}\in\dsbh$, let $F$ and $G$ be functions
on $\rr^{n+1}_\zz$ defined by setting, for all $x\in\rn$ and $j\in\zz$,
$F(x, 2^{-j})\equiv
\sum_{Q\in\cD_j(\rn)}|\lambda_Q|\widetilde{\chi}_Q$ and
$G(x, 2^{-j})\equiv \sum_{P\in\cD_j(\rn)}|t_P|\widetilde{\chi}_P$. Since
$$\|F\|_{B\dot{T}^{s,\tau}_{p,q}(\R^{n+1}_\zz)}\sim
\|\lambda\|_{\dsbh}$$
and
$\|G\|_{B\dot{W}^{-s,\tau}_{p',q'}(\R^{n+1}_\zz)}\sim
\|t\|_{\dot{b}_{p',q'}^{-s,\tau}(\rn)}$, where
$B\dot{T}^{s,\tau}_{p,q}(\R^{n+1}_\zz)$ and
$B\dot{W}^{-s,\tau}_{p',q'}(\R^{n+1}_\zz)$ are tent spaces
introduced in \cite[Definition 5.2]{yy2}, by the duality of
tent spaces obtained in \cite[Theorem 5.1]{yy2} that
$(B\dot{T}^{s,\tau}_{p,q}(\R^{n+1}_\zz))^*=
B\dot{W}_{p',q'}^{-s,\tau}(\R^{n+1}_\zz)$, we have
\begin{eqnarray*}
\l|\sum_{Q\in\cD(\rn)}\lambda_Q\overline{t}_Q\r| &&\le
\sum_{j\in\zz}\int_\rn\sum_{Q\in\cD_j(\rn)}\sum_{P\in\cD_j(\rn)}
|\lambda_Q|\widetilde{\chi}_Q(x)|t_P|\widetilde{\chi}_P(x)\,dx\\
&&=\sum_{j\in\zz}\int_\rn F(x,2^{-j})G(x,2^{-j})\,dx\ls
\|F\|_{B\dot{T}^{s,\tau}_{p,q}(\R^{n+1}_\zz)}
\|G\|_{B\dot{W}_{p',q'}^{-s,\tau}(\R^{n+1}_\zz)}\\
&&\sim
\l\|\lambda\r\|_{\dsbh}\l\|t\r\|_{\dot{b}_{p',q'}^{-s,\tau}(\rn)},
\end{eqnarray*}
which implies that $\dot{b}_{p',q'}^{-s,\tau}(\rn)\hookrightarrow(\dsbh)^*$.

Conversely, since sequences with finite non-vanishing elements are
dense in $\dsbh$, we know that every $L\in(\dsbh)^*$ is of the form
$\lz\mapsto\sum_{Q\in\cD(\rn)}\lambda_Q\overline{t}_Q$ for a certain
$t=\{t_Q\}_{Q\in\cD(\rn)}\subset\mathbb{C}$. It remains to show that
$\|t\|_{\dot{b}_{p',q'}^{-s,\tau} (\rn)}\ls\|t\|_{(\dsbh)^*}$.

Fix $P\in\cD(\rn)$ and $a\in\rr$. For $j\ge j_P$, let $X_j$ be the
set of all $Q\in \cD_j(\rn)$ satisfying $Q\subset P$ and let
$\mu$ be a measure on $X_j$ such that the $\mu$-measure of the
``point" $Q$ is $|Q|/|P|^{\tau a}$. Also, let $l^q_P$ denote the
set of all $\{a_j\}_{j\ge j_P}\subset \C$ with
$\|\{a_j\}_{j\ge j_P}\|_{l^q_P}\equiv (\sum_{j=j_P}^\fz
|a_j|^q)^{1/q}$ and $l^{q}_P(l^{p}(X_j, d\mu))$ denote the set
of all $\{a_{Q,j}\}_{Q\in\cD_j(\rn),\,Q\subset P,\,j\ge
j_P}\subset\C$ with
$$\|\{a_{Q,j}\}_{Q\in\cD_j(\rn),\,Q\subset P,\,j\ge
j_P}\|_{l^q_P(l^{p}(X_j, d\mu))}\equiv
\l(\sum_{j=j_P}^\fz\l[\sum_{Q\in\cD_j(\rn),\,Q\subset P}
|a_{Q,j}|^p\frac{|Q|}{|P|^{\tau a}}\r]^{\frac qp}\r)^{1/q}.$$
It is easy to see that the dual space of $l^{q}_P(l^{p}(X_j, d\mu))$ is
$l^{q'}_P(l^{p'}(X_j, d\mu))$; see \cite[p.\,177]{t83}. Via this
observation and the already proved conclusion of this proposition, we see that
\begin{eqnarray*}
&&\frac1{|P|^\tau}\l\{\sum_{j=j_P}^\fz\l[\sum_{Q\in\cD_j(\rn),\,Q\subset
P}\l(|Q|^{-\frac
sn-\frac12}|t_Q|\r)^{p'}|Q|\r]^{\frac{q'}{p'}}\r\}^{\frac1{q'}}\\
&&\hs=\|\{|Q|^{-\frac sn-\frac12}|t_Q|\}_{Q\in\cD_j(\rn),\,Q\subset
P,\,j\ge
j_P}\|_{l^{q'}_P(l^{p'}(X_j, d\mu))}\\
&&\hs=\sup_{\|\{\lambda_Q\}_{Q\in\cD_j(\rn),\,Q\subset P,\,j\ge
j_P}\|_{l^{q}_P(l^{p}(X_j,
d\mu))}\le1}\l|\sum_{j=j_P}^\fz\sum_{Q\in\cD_j(\rn),\,Q\subset
P}\lambda_Q|Q|^{-\frac sn-\frac12}|t_Q||Q|/|P|^{\tau p'}\r|\\
&&\hs\le \sup_{\|\{\lambda_Q\}_{Q\in\cD_j(\rn),\,Q\subset P,\,j\ge
j_P}\|_{l^{q}_P(l^{p}(X_j,
d\mu))}\le1}\|t\|_{(\dsbh)^*}\\
&&\hs\hs\times\l\|\{\lambda_Q|Q|^{-\frac sn-\frac12}|Q|/|P|^{\tau
p'}\}_{Q\in\cD_j(\rn),\,Q\subset P,\,j\ge j_P}\r\|_{\dsbh}.
\end{eqnarray*}

To finish the proof of this proposition, it suffices to show that
$$\l\|\{\lambda_Q|Q|^{-\frac sn-\frac12}|Q|/|P|^{\tau
p'}\}_{Q\in\cD_j(\rn),\,Q\subset P,\,j\ge j_P}\r\|_{\dsbh}\ls 1$$
for all sequences $\lambda$ satisfying
$\|\{\lambda_Q\}_{Q\in\cD_j(\rn),\,Q\subset P,\,j\ge
j_P}\|_{l^{q}_P(l^{p}(X_j, d\mu))}\le1$. In fact, let $B\equiv
B(c_P, \sqrt{n}\ell(P))$ and $\omega$ be as in the proof of
\cite[Lemma 4.1]{yy1} associated with $B$, then $\omega$ satisfies
\eqref{1.1} and for all $x\in P$ and $j\ge j_P$, $[\omega(x,
2^{-j})]^{-1}\sim [\ell(P)]^{n\tau}$. We then obtain that
\begin{eqnarray*}
&&\l\|\{\lambda_Q|Q|^{-\frac sn-\frac12}|Q|/|P|^{\tau
p'}\}_{Q\in\cD_j(\rn),\,Q\subset P,\,j\ge j_P}\r\|_{\dsbh}\\
&&\hs\ls\l\{\sum_{j=j_P}^\fz
2^{jsq}\l[\sum_{Q\in\cD_j(\rn),\,Q\subset P}|Q|^{-\frac
p2}\l(|\lambda_Q||Q|^{-\frac sn-\frac12}\frac{|Q|}{|P|^{\tau
p'}}\r)^p\int_Q[\omega(x,2^{-j})]^{-p}\,dx\r]^{\frac
qp}\r\}^{\frac1q}\\
&&\hs\sim \l\{\sum_{j=j_P}^\fz
\l[\sum_{Q\in\cD_j(\rn),\,Q\subset P}|\lambda_Q|^p|Q|/|P|^{\tau p'}
\r]^{\frac
qp}\r\}^{\frac1q}\\
&&\hs\sim\|\{\lambda_Q\}_{Q\in\cD_j(\rn),\,Q\subset P,\,j\ge
j_P}\|_{l^{q}_P(l^{p}(X_j, d\mu))}\ls1,
\end{eqnarray*}
which completes the proof of Proposition \ref{p2.1}.
\end{proof}

\begin{remark}\label{r2.2}
By Proposition \ref{p2.1} and the $\varphi$-transform
characterizations of the spaces $\dah$ in Theorem \ref{t2.1} and
$\dat$ in \cite[Theorem 3.1]{yy2}, we also obtain the duality that
$(\dah)^*= \dot{A}_{p',q'}^{-s,\tau}(\rn)$. This gives other proofs
of these conclusions, which are different from those in
\cite[Section 5]{yy1} and \cite[Section 6]{yy2}.
\end{remark}

Applying Theorem \ref{t2.1}, we establish the following Sobolev-type
embedding properties of $\dah$. For the corresponding results on
$\db$ and $\df$, see \cite[p.\,129]{t83}.

\begin{proposition}\label{p2.2}
Let $1<p_0< p_1<\fz$ and $-\fz<s_1< s_0<\fz$.
Assume in addition that $s_0-n/p_0=s_1-n/p_1$.

(i) If $q\in[1,\,\fz)$ and $\tau\in[0, \min\{\frac1{(p_0\vee q)'},
\frac1{(p_1\vee q)'}\}]$ such that $\tau(p_0\vee q)'=\tau(p_1\vee
q)'$, then $B{\dot H}^{s_0,\,\tau}_{p_0,\,q}(\rn)\hookrightarrow
B{\dot H}^{s_1,\,\tau}_{p_1,\,q}(\rn)$.

(ii) If $q, r\in(1,\fz)$ and $\tau\in[0, \min\{\frac1{(p_0\vee
r)'}, \frac1{(p_1\vee q)'}\}]$ such that $\tau (p_0\vee r)'\le \tau
(p_1\vee q)'$, then $F{\dot
H}^{s_0,\,\tau}_{p_0,\,r}(\rn)\hookrightarrow F{\dot
H}^{s_1,\,\tau}_{p_1,\,q}(\rn)$.
\end{proposition}

\begin{proof}
By Theorem \ref{t2.1} and similarity, it suffices to prove the
corresponding conclusions on sequence spaces $\dsfh$, namely, to
show that $\|t\|_{f{\dot H}^{s_1,\,\tau}_{p_1,\,q}(\rn)}\ls
\|t\|_{f{\dot H}^{s_0,\,\tau}_{p_0,\,r}(\rn)}$ for all
$t\in f{\dot H}^{s_0,\,\tau}_{p_0,\,r}(\rn)$. When $\tau=0$, this
is a classic conclusion on Triebel-Lizorkin spaces.

In the case when $\tau>0$, we have $(p_0\vee r)'\le (p_1\vee q)'$.
Let $t\in fH^{s_0,\tau}_{p_0,r}(\rn)$ and $\omega$ satisfy
\begin{equation}\label{2.7}
\int_\rn [N\omega(x)]^{(p_0\vee r)'}\,d H^{n\tau (p_0\vee
r)'}(x)\le1
\end{equation}
and
$$\l\{\int_{\rn}\l[\sum_{j\in\zz}^\fz\sum_{Q\in\cD_j(\rn)}
|Q|^{-\frac{s_0r}{n}-\frac{r}{2}}
|t_Q|^{r}\chi_Q(x)[\omega(x,2^{-j})]^{-r}\r]^{p_0/r}\,dx\r\}^{1/p_0}
\ls\|t\|_{f\dot{H}^{s_0,\tau}_{p_0,r}(\rn)}.$$ For all
$(x,t)\in\rr^{n+1}_+$, we set $ \wz\omega(x,s)\equiv \sup\{
\omega(y,s) \, : \, y\in\rn,\ |y-x|<\sqrt{n}s \}.$ Then by the
argument in \cite[Lemma 5.2]{yy1}, we know that a constant multiple
of $\wz\omega$ also satisfies \eqref{2.7}. Since
$(p_0\vee r)'\le (p_1\vee q)'$, Remark \ref{r2.1}(i)
tells us that $\wz\omega$ satisfies
$$\int_\rn [N\omega(x)]^{(p_1\vee q)'}\,d H^{n\tau (p_1\vee
q)'}(x)\ls1.$$ For all $Q$ with $\ell(Q)=2^{-j}$, set
$\wz{t}_Q\equiv |t_Q|\sup_{y\in Q}\{[\wz\omega(y,2^{-j})]^{-1}\}$.
Observe that for all $x\in Q$ with $\ell(Q)=2^{-j}$,
$[\wz\omega(x,2^{-j})]^{-1}\ls \inf_{y\in Q}
[\omega(y,2^{-j})]^{-1}$, and hence, $ \sup_{x\in
Q}[\wz\omega(x,2^{-j})]^{-1}\ls \inf_{y\in Q}
[\omega(y,2^{-j})]^{-1}. $ This observation together with $p_0<p_1$,
$s_0-n/p_0=s_1-n/p_1$ and the corresponding embedding property for
Triebel-Lizorkin spaces (see, for example, \cite[Theorem 2.7.1]{t83})
yields that
\begin{eqnarray*}
&&\l\{\int_{\rn}\l[\sum_{j\in\zz}\sum_{Q\in\cD_j(\rn)}
|Q|^{-\frac{s_1q}{n}-\frac{q}{2}}
|t_Q|^{q}\chi_Q(x)[\wz\omega(x,2^{-j})]^{-q}\r]^{p_1/q}\,dx\r\}^{1/p_1}\\
&&\hs\le\l\{\int_{\rn}\l[\sum_{j\in\zz}\sum_{Q\in\cD_j(\rn)}
|Q|^{-\frac{s_1q}{n}-\frac{q}{2}} |t_Q|^{q}\chi_Q(x)\sup_{y\in
Q}\l\{[\wz\omega(y,2^{-j})]^{-q}\r\}\r]^{p_1/q}\,dx\r\}^{1/p_1}\\
&&\hs=\l\{\int_{\rn}\l[\sum_{j\in\zz}\sum_{Q\in\cD_j(\rn)}
|Q|^{-\frac{s_1q}{n}-\frac{q}{2}}
|\wz{t}_Q|^{q}\chi_Q(x)\r]^{p_1/q}\,dx\r\}^{1/p_1}=\|\wz{t}\|_{\dot f^{s_1}_{p_1,q}(\rn)}
\ls \|\wz{t}\|_{\dot f^{s_0}_{p_0,r}(\rn)}\\
&&\hs\sim \l\{\int_{\rn}\l[\sum_{j\in\zz}\sum_{Q\in\cD_j(\rn)}
|Q|^{-\frac{s_0r}{n}-\frac{r}{2}}
|\wz{t}_Q|^{r}\chi_Q(x)\r]^{p_0/r}\,dx\r\}^{1/p_0}\\
&&\hs\sim \l\{\int_{\rn}\l[\sum_{j\in\zz}\sum_{Q\in\cD_j(\rn)}
|Q|^{-\frac{s_0r}{n}-\frac{r}{2}} |t_Q|^{r}\chi_Q(x)\sup_{y\in
Q}\l\{[\wz\omega(y,2^{-j})]^{-r}\r\}\r]^{p_0/r}\,dx\r\}^{1/p_0}\\
&&\hs\ls \l\{\int_{\rn}\l[\sum_{j\in\zz}\sum_{Q\in\cD_j(\rn)}
|Q|^{-\frac{s_0r}{n}-\frac{r}{2}}
|t_Q|^{r}\chi_Q(x)[\omega(x,2^{-j})]^{-r}\r]^{p_0/r}\,dx\r\}^{1/p_0}
\ls\|t\|_{f\dot{H}^{s_0,\tau}_{p_0,r}(\rn)};
\end{eqnarray*}
see \cite[p.\,38]{fj} for the definition
of the sequence spaces $\dot f^s_{p,q}(\rn)$.
Therefore,
$\|t\|_{f\dot{H}^{s_1,\tau}_{p_1,q}(\rn)}\ls\|t\|_{f\dot{H}^{s_0,\tau}_{p_0,r}(\rn)}$,
which completes the proof of Proposition \ref{p2.2}.
\end{proof}

When $\tau=0$, Proposition \ref{p2.2} recovers the corresponding
results on $\db$ and $\df$ in \cite[p.\,129]{t83}, which are known
to be sharp; see \cite[p.\,207]{t95}. At the end of this section, we
further show that the restriction that $\tau(p_0\vee
q)'=\tau(p_1\vee q)'$ in Proposition \ref{p2.2}(i) is also
sharp. To see this, we need the following geometrical observation on
the Hausdorff capacity.

\begin{lemma}\label{l2.4}
Let $d\in(0, n]$. Suppose that $\{E_j\}_{j=1}^\fz$
are given subsets of $\rn$ such that
$E_j \subset B((A_j,0,\cdots,0),n)$,
where $\{A_j\}_{j=1}^\fz$ is an increasing sequence of natural
numbers satisfying that $A_1\ge 10$ and for all $j,\,l\in\nn$,
$A_{j+l}-A_j\ge 4nl^{1/d}$. Then $H^d(\cup_{j=1}^\infty E_j)$ and
$\sum_{j=1}^\infty H^d(E_j)$ are equivalent.
\end{lemma}

\begin{proof}
The inequality $H^d(\cup_{j=1}^\infty E_j) \le \sum_{j=1}^\infty
H^d(E_j)$ is trivial. Let us prove the reverse inequality. To this
end, let us first notice the following geometric
observation that when a ball $B\equiv(x_B,r_B)$
intersects $E_j$ and $E_{j+l}$ for some $j,\,l\in\nn$, then $2B$
engulfs $E_j,\,E_{j+1},\cdots,E_{j+l}$. Thus, $4r_B$ is greater than
$A_{j+l}-A_j$ and hence, $r_B^d\ge((A_{j+l}-A_j)/4)^d\ge ln^d$.
Therefore, instead of using $B$ we can use
$B((A_j,0,\cdots,0),n)$,\,$\cdots$,\,
$B((A_{j+l},0,\cdots,0),n)$ to cover $E_j$ and $E_{j+l}$.
Notice that $\{B((A_j,0,\cdots,0),n)\}_{j=1}^\fz$ are disjoint.
Based on these observations, without loss of generality,
we may assume, in estimating $H^d(\cup_{j=1}^\infty E_j)$, that each ball
in the ball covering meets only one $E_j$. From this, it is
easy to follow that $H^d(\cup_{j=1}^\infty E_j) \gtrsim
\sum_{j=1}^\infty H^d(E_j)$, which completes the proof of Lemma
\ref{l2.4}.
\end{proof}

\begin{lemma}\label{l2.5}
Let $s\in\rr$, $p\in(1,\fz)$, $q\in[1,\fz)$,
$\tau\in(0,\frac1{(p\vee q)'}]$ and  $\{A_k\}_{k=1}^\fz$ be as in Lemma
\ref{l2.4} such that $Q_k\equiv (A_k,0,\cdots,0)+2^{-k}[0,1)^n \in
\cD_k(\rn)$ for all $k\in\nn$ (The existence of $\{A_k\}_{k=1}^\fz$
is obvious). Define $t_j\equiv\{(t_j)_{Q}\}_{Q
\in \cD(\rn)}$ so that $(t_j)_{Q}\equiv 2^{-\frac{kn}2-k(s-\frac
np)}$ if $Q=Q_k$ and $k\in\{1,\cdots, j\}$, $(t_j)_{Q}\equiv0$
otherwise. Then for all $j \in \N$, $\| t_j
\|_{b\dot{H}_{p,q}^{s,\tau}(\rn)}$ is equivalent to
$j^{\frac{1}{q}+\frac{1}{(p \vee q)'}}$ and $\| t_j
\|_{f\dot{H}_{p,q}^{s,\tau}(\rn)}$ is equivalent to
$j^{\frac1p+\frac{1}{(p \vee q)'}}$.
\end{lemma}

\begin{proof}
For the Besov-Hausdorff space, let us minimize
\[
\left( \sum_{k=1}^j 2^{ksq} \left\|
|(t_j)_{Q_k}|\widetilde{\chi}_{Q_k}[\omega(\cdot,2^{-k})]^{-1}
\right\|_{L^p(\R^n)}^q \right)^\frac1q
\]
under the condition \eqref{1.1}. By the definition of $t_j$ and the
assumption on $\omega$ in Definition \ref{d2.1}, we may assume that
$\omega\equiv0$ outside
$\cup_{k=1}^j(Q_{0,(A_k,0,\cdots,0)}\times\{2^{-k}\})$ and for all
$Q\in\cD_k(\rn)$, $Q\subset Q_{0,(A_k,0,\cdots,0)}$ and
$k\in\{1,\cdots,j\}$, $\sup_{x\in Q}\omega(x,2^{-k})=\sup_{x\in
Q_k}\omega(x,2^{-k})$, where $Q_{0,(A_j,0,\cdots,0)}\equiv
(A_j,0,\cdots,0)+[0,1)^n\in\cD_0(\rn)$. Also, by an observation
similar to \cite[Lemma 6.2]{yy2}, we can replace $\omega$ with the
maximal function $\widetilde{\omega}$ given by
$\widetilde{\omega}(x,2^{-k})\equiv\sup_{y \in
Q_{k,x}}\omega(y,2^{-k}),$ where $k\in\{1,\cdots,j\}$ and $Q_{k,x}
\in \cD_k(\rn)$ is a unique cube containing $x$. This construction
implies that $\wz\omega$ equals a constant on
$Q_{0,(A_k,0,\cdots,0)}$ for each $k\in\{1,\cdots,j\}$, namely,
$\wz\omega(\cdot,2^{-k})\equiv\alpha_k
\chi_{Q_{0,(A_k,0,\cdots,0)}}$. Notice that if
$N\widetilde{\omega}(x) \not=0$, then $x\in B((A_k,0,\cdots,0), n)$
for some $k\in\{1,\cdots,j\}$. This combined with Lemma \ref{l2.4}
yields that
\begin{eqnarray*}
&&\int_{\R^n} [N\widetilde{\omega}(x)]^{(p \vee q)'} \,dH^{n\tau(p\vee q)'}(x)\\
&&\hs=\int_0^\fz H^{n\tau(p\vee q)'}\left(\left\{x\in
\left(\bigcup_{k=1}^jB((A_k,0,\cdots,0), n)\r):\,[N\widetilde{\omega}(x)]^{(p \vee q)'}
>\lambda\r\}\r)\,d\lambda\\
&&\hs\sim\sum_{k=1}^j\int_0^\fz H^{n\tau(p\vee q)'}\left(\left\{x\in
B((A_k,0,\cdots,0), n):\,[N\widetilde{\omega}(x)]^{(p \vee q)'}
>\lambda\r\}\r)\,d\lambda\\
&&\hs\sim \sum_{k=1}^j \int_{B((A_k,0,\cdots,0), n)}
[N\widetilde{\omega}(x)]^{(p \vee q)'} \,d H^{n\tau(p\vee
q)'}(x)\sim \sum_{k=1}^j (\az_k)^{(p \vee q)'}.
\end{eqnarray*}
On the other hand,
\[
\left( \sum_{k=1}^j 2^{ksq} \left\|
|(t_j)_{Q_k}|\widetilde{\chi}_{Q_k}[\wz\omega(\cdot,2^{-k})]^{-1}
\right\|_{L^p(\R^n)}^q \right)^\frac1q =\left[\sum_{k=1}^j
(\az_k)^{-q}\right]^\frac{1}{q}.
\]
In summary (modulo a multiplicative constant), we need to minimize
$(\sum_{k=1}^j (\az_k)^{-q})^\frac{1}{q}$ under the condition
$\sum_{k=1}^j (\az_k)^{(p \vee q)'}\ls1.$ This can be achieved as
follows\,: By using the geometric mean, we have
\begin{align*}
\left(\sum_{k=1}^j (\az_k)^{-q}\right)^\frac{1}{q} &\gs
\left(\sum_{k=1}^j (\az_k)^{-q}\right)^\frac{1}{q}
\left(\sum_{k=1}^j (\az_k)^{(p \vee q)'}\right)^{\frac{1}{(p \vee q)'}}\\
&\gs \left(j \sqrt[j]{\prod_{k=1}^j(\az_k)^{-q}}\right)^\frac{1}{q}
\left(j \sqrt[j]{\prod_{k=1}^j(\az_k)^{(p \vee
q)'}}\right)^\frac{1}{(p \vee q)'} \sim j^{\frac{1}{q}+\frac{1}{(p
\vee q)'}}.
\end{align*}
In particular, $[\sum_{k=1}^j (\az_k)^{-q}]^\frac{1}{q}\sim
j^{\frac{1}{q}+\frac{1}{(p \vee q)'}}$ when $\sum_{k=1}^j
(\az_k)^{(p \vee q)'}\sim1$ and the $\az_k$'s are identical.
Thus, for all $j \in \N$, $\|t_j\|_{b\dot{H}_{p,q}^{s,\tau}(\rn)}
\sim j^{\frac{1}{q}+\frac{1}{(p \vee q)'}}$.

For Triebel-Lizorkin-Hausdorff space, similarly to the above arguments, we see
that
\begin{eqnarray*}
&&\left(\int_\rn\left[\sum_{k=1}^j|Q_k|^{-(s/n+1/2)q}|(t_j)_{Q_k}|^q\chi_{Q_k}(x)
[\wz\omega(x,2^{-k})]^{-q}
\right]^{p/q}\,dx \right)^\frac1p\\
&&\hs=\left(\int_\rn\sum_{k=1}^j|Q_k|^{-(s/n+1/2)p}|(t_j)_{Q_k}|^p\chi_{Q_k}(x)
(\az_k)^{-p}\,dx \right)^\frac1p=\left[\sum_{k=1}^j
(\az_k)^{-p}\r]^{1/p}.
\end{eqnarray*}
Applying the geometric mean again, we have
\begin{align*}
\left(\sum_{k=1}^j(\az_k)^{-p}\right)^\frac{1}{p} &\gs
\left(\sum_{k=1}^j (\az_k)^{-p}\right)^\frac{1}{p}
\left(\sum_{k=1}^j (\az_k)^{(p \vee q)'}\right)^{\frac{1}{(p \vee q)'}}\\
&\gs \left(j \sqrt[j]{\prod_{k=1}^j(\az_k)^{-p}}\right)^\frac{1}{p}
\left(j \sqrt[j]{\prod_{k=1}^j(\az_k)^{(p \vee
q)'}}\right)^\frac{1}{(p \vee q)'} \sim j^{\frac1p+\frac{1}{(p \vee
q)'}}.
\end{align*}
In particular, $[\sum_{k=1}^j (\az_k)^{-p}]^\frac{1}{p}\sim
j^{\frac{1}{p}+\frac{1}{(p \vee q)'}}$ when $\sum_{k=1}^j
(\az_k)^{(p \vee q)'}\sim1$ and the $\az_k$'s are identical, which
implies that for all $j \in \N$, $\| t_j\|_{f\dot{H}_{p,q}^{s,\tau}(\rn)}
\sim j^{\frac1p+\frac{1}{(p \vee q)'}}$.
This finishes the proof of Lemma \ref{l2.5}.
\end{proof}

\begin{proposition}\label{p2.3}
Let $s,\,\tau,\,p_0,\,p_1,\,q,\,r$ be as in Proposition \ref{p2.2}.

(i) If $ b\dot{H}_{p_0,q}^{s_0,\tau} \hookrightarrow
b\dot{H}_{p_1,q}^{s_1,\tau}$, then $\tau(p_0 \vee q)'=\tau(p_1 \vee
q)'$.

(ii) If $ f\dot{H}_{p_0,r}^{s_0,\tau} \hookrightarrow
f\dot{H}_{p_1,q}^{s_1,\tau}$, then $\tau(p_0 \vee r)'\le \tau(p_1
\vee q)'+\tau(\frac1{p_0}-\frac1{p_1})(p_0 \vee r)'(p_1 \vee q)'$.
\end{proposition}

\begin{proof}
By similarity, we only consider the Besov-Hausdorff space. Let $t_j$ be as in
Lemma \ref{l2.5} with $s$, $p$ replaced, respectively, by $s_0$ and
$p_0$. Since $s_0-n/p_0=s_1-n/p_1$, by Lemma \ref{l2.5}, we have $\|
t_j \|_{b\dot{H}_{p_0,q}^{s_0,\tau}} \sim
j^{\frac{1}{q}+\frac{1}{(p_0 \vee q)'}}$ and $\| t_j
\|_{b\dot{H}_{p_1,q}^{s_1,\tau}} \sim j^{\frac{1}{q}+\frac{1}{(p_1
\vee q)'}}$ for all $j \in \N$, which together with $
b\dot{H}_{p_0,q}^{s_0,\tau} \hookrightarrow
b\dot{H}_{p_1,q}^{s_1,\tau}$ implies that
$j^{\frac{1}{q}+\frac{1}{(p_1 \vee q)'}} \ls
j^{\frac{1}{q}+\frac{1}{(p_0 \vee q)'}}$ for all $j \in \N$.
Therefore, $(p_0 \vee q)' \le (p_1 \vee q)'$. Meanwhile it is
trivial that $(p_0 \vee q)' \ge (p_1 \vee q)'$ since $p_1>p_0$. We
then have $(p_0 \vee q)' = (p_1 \vee q)'$. This finishes the proof
of Proposition \ref{p2.3}.
\end{proof}

\begin{remark}\label{r2.3}
Comparing Proposition \ref{p2.2} herein with \cite[Proposition 3.3]{yy2}
on the space $\dbt$, we see that the restriction $\tau(p_0\vee q)'=\tau(p_1\vee q)'$
in Proposition \ref{p2.2}(i) is additional.
To be surprising, Proposition \ref{p2.3}(i) implies that this restriction
is also necessary, and sharp in this sense. However, it is still unclear if
the restriction $\tau(p_0\vee r)'\le\tau(p_1\vee q)'$
in Proposition \ref{p2.2}(ii) can be replaced by
the restriction $\tau(p_0 \vee r)'\le \tau(p_1
\vee q)'+\tau(\frac1{p_0}-\frac1{p_1})(p_0 \vee r)'(p_1 \vee q)'$.
\end{remark}

\section{Smooth atomic and molecular decompositions\label{s3}}

\hskip\parindent We begin with considering the boundedness of almost
diagonal operators on $\dsah$, which is applied to establish the
smooth atomic and molecular decomposition characterizations of
$\dah$. We remark that the corresponding results in $\dsat$ and
$\dat$ were already obtained in \cite[Section 4]{yy2}.

\begin{definition}\label{d3.1}
Let $p\in(1,\fz), \, q\in[1,\fz), \, s \in \R$, $\tau\in[0,
\frac{1}{(p \vee q)'}]$ and $\varepsilon\in(0,\fz)$. For all
$Q,\,P\in\cD(\rn)$, define
\[
\omega_{QP}(\varepsilon) \equiv
\left(\frac{\ell(Q)}{\ell(P)}\right)^s
\left(1+\frac{|x_P-x_Q|}{\max(\ell(Q),\ell(P))}
\right)^{-n-\varepsilon} \min\left(
\left(\frac{\ell(P)}{\ell(Q)}\right)^{\frac{n+\varepsilon}{2}},
\left(\frac{\ell(Q)}{\ell(P)}\right)^{\frac{n+\varepsilon}{2}}
\right).
\]
An operator $A$ associated with a matrix $\{a_{QP}\}_{Q,P \in
\cD(\rn)}$, namely, for all sequences
$t=\{t_Q\}_{Q\in\cD(\rn)}\subset \mathbb{C}$, $At\equiv
\{(At)_Q\}_{Q \in \cD(\rn)}\equiv \{\sum_{P \in
\cD(\rn)}a_{QP}t_P\}_{Q \in \cD(\rn)}$, is called
\textit{$\varepsilon$-almost diagonal on $\dsah$}, if the matrix
$\{a_{QP}\}_{Q,P \in \cD(\rn)}$ satisfies
$$\sup_{Q,\,P \in \cD(\rn)}
|a_{QP}|/\omega_{QP}(\varepsilon)<\infty.$$
\end{definition}

We remark that any $\varepsilon$-almost diagonal operator on $\dsah$
is also an almost diagonal operator introduced by Frazier and
Jawerth in \cite{fj} with $J\equiv n$. Moreover, Frazier and Jawerth
proved that all almost diagonal operators are bounded on
$\dot{b}^s_{p,q}(\rn)$ and $\dot{f}^s_{p,q}(\rn)$, which are the
corresponding sequence spaces of $\dot{B}^s_{p,q}(\rn)$ and
$\dot{F}^s_{p,q}(\rn)$; see \cite{fj88,fj,fjw}. These results when
$p\in(1,\fz)$ and $q\in[1,\fz)$ are generalized into the following
conclusions.

\begin{theorem}\label{t3.1}
Let $p\in(1,\fz), \, q\in[1,\fz), \, s \in \R$,
$\varepsilon\in(0,\fz)$ and $\tau\in[0, \frac{1}{(p \vee q)'}]$.
Then all the $\varepsilon$-almost diagonal operators on
$a\dot{H}_{p,q}^{s,\tau}(\rn)$ are bounded if $\varepsilon>2n\tau$.
\end{theorem}

To prove this theorem, we need some technical lemmas.

\begin{lemma}\label{l3.1}
Let $d\in(0,n]$ and $\Omega$ be an open set in $\rn$ such that
$\Omega=\cup_{j=1}^\fz B_j$, where
$\{B_j\}_{j=1}^\infty\equiv\{B(X_j,R_j)\}_{j=1}^\infty$ is a
countable collection of balls. Define
\begin{eqnarray*}
&&H^d(\Omega,\{B_j\}_{j=1}^\infty)\\
&&\hs\equiv \inf \left\{
\sum_{k=1}^\infty r^d_k : \, \Omega \subset \bigcup_{k=1}^\infty
B(x_k,r_k), \, B(x_k,r_k) \supset B_j\ \mathrm{if}\ B_j \cap
B(x_k,r_k) \ne \emptyset \right\}.
\end{eqnarray*}
Then there exists a positive constant $C$, independent of $\Omega$,
$\{B_j\}_{j=1}^\infty$ and $d$, such that
\[
H^d(\Omega) \le H^d(\Omega,\{B_j\}_{j=1}^\infty) \le C (46)^d
H^d(\Omega).
\]
\end{lemma}

\begin{proof} The first inequality is trivial. We only need to prove
the second one. Without loss of generality, we may assume $\sup_{j
\in \nn}R_j<\infty$. By the well-known $(5r)$-covering lemma (see,
for example, \cite[Theorem 2.19]{du}), there exists a subset $J^*$
of $\nn$ such that $\cup_{j =1}^\fz(3B_j) \subset \cup_{j \in
J^*}(15B_j)$ and $\chi_{j \in J^*}\chi_{(3B_j)} \le 1$. Furthermore,
by its construction, if $B_{j'}$, $j'\in\nn$, intersects $B_j$ for
some $j \in J^*$, we have that $(3B_{j'}) \subset (15B_j)$.

Let $\{B(x_k,r_k)\}_{k \in \nn}$ be a collection of balls such that
$ \Omega \subset \cup_{k=1}^\infty B(x_k,r_k)$ and $\sum_{k=1}^\fz
r_k^d \le 2H^d(\Omega)$. Set
$$
K_1\equiv\{ k \in \nn \, : \, \mbox{When}\ B(x_k,45r_k)\cap
B_j\neq\emptyset\ \mbox{for\, any}\,j\in\nn, \mbox{ then } r_k\ge
135R_j\}
$$
and
$
J_1\equiv\{ j \in \nn \, : \, B_j\cap B(x_k,45r_k)\neq\emptyset
\mbox{ for some } k \in K_1 \}.
$
Also define $J_2\equiv(\nn \setminus J_1)$ and $K_2\equiv(\nn
\setminus K_1)$. We remark that if $k \in K_2$, then there exists $j
\in J_2$ such that $B_j\cap B(x_k,45r_k)\neq\emptyset$ and $135R_j
> r_k$. Notice that $B_j \subset \Omega \subset (\cup_{k=1}^\infty
B(x_k,r_k))$. Hence, for each $j \in J_2$, we have $B_j \subset
(\cup_{k \in K_2, \, B(x_k,r_k) \cap B_j \ne \emptyset}B(x_k,r_k)),$
and then, by $d\le n$ and the monotonility of $l^{\frac dn}$, we see
that
\begin{eqnarray*}
\sum_{k \in K_2}r_k^d&&\sim\sum_{k\in K_2}|B(x_k, r_k)|^{\frac dn}
\gs \sum_{j \in J^* \cap J_2} \sum_{k \in K_2, \, B_j \cap
B(x_k,45r_k)
\ne \emptyset}|B(x_k, r_k)|^{\frac dn}\\
&&\gs \sum_{j \in J^* \cap J_2}\l( \sum_{k \in K_2, \, B_j \cap
B(x_k,45r_k) \ne \emptyset}|B(x_k, r_k)|\r)^{\frac dn}\gs \sum_{j
\in J^* \cap J_2}R_j^d,
\end{eqnarray*}
which further yields that
\[
\sum_{k \in K_1}r_k^d + \sum_{j \in J^* \cap J_2}R_j^d \ls \sum_{k
\in K}r_k^d.
\]

On the other hand, we have
\begin{align*}
\Omega &\subset \bigcup_{j=1}^\fz B_j\subset \bigcup_{j \in
J^*}(15B_j)= \Bigg\{\bigcup_{j \in J^* \cap J_1}(15B_j)\Bigg\}
\bigcup\Bigg\{
\bigcup_{j \in J^* \cap J_2}(15B_j)\Bigg\}\\
&\subset \Bigg\{\bigcup_{k \in K_1}B(x_k,46r_k)\Bigg\} \bigcup
\Bigg\{\bigcup_{j \in J^* \cap J_2}(15B_j)\Bigg\}.
\end{align*}
Notice that for $k\in K_1$, $B(x_k,45r_k)$ meets $B_j$ for some $j
\in \nn$ gives us $r_k \ge 135R_j$, which further implies that
$B(x_k,46r_k) \supset B_j$. Also, for $j \in J^*$ and $j' \in \nn$,
if $B_j\cap B_{j'}\neq \emptyset$, then $(15B_{j})\supset B_{j'}$.
As a result, we conclude that $\{B(x_k,46r_k)\}_{k \in K_1} \cup
\{15B_j\}_{j \in J^*\cap J_2}$ is the desired covering of $\Omega$
and hence,
\[
H^d(\Omega,\{B_j\}_{j=1}^\infty) \le \sum_{k \in K_1}(46r_k)^d
+\sum_{j \in J^*\cap J_2}(15R_j)^d \ls (46)^d H^d(\Omega),
\]
which completes the proof of Lemma \ref{l3.1}.
\end{proof}

Applying Lemma \ref{l3.1}, we have the following conclusion.

\begin{lemma}\label{l3.2}
Let $\beta \in[1,\fz)$, $\lambda\in(0,\fz)$ and $\omega$ be a
nonnegative Borel measurable function on $\rr^{n+1}_+$. Then there
exists a positive constant $C$, independent of $\beta$, $\omega$ and
$\lambda$, such that
$$H^d\l(\{x\in\rn:\,N_\beta \omega(x)>\lambda\}\r) \le C\beta^d
H^d\l(\{x\in\rn:\,N \omega(x)>\lambda \}\r),$$ where $N_\beta
\omega(x)\equiv \sup_{|y-x|<\beta t}\omega(y, t)$.
\end{lemma}

\begin{proof}
Observe that
\[
\{x\in\rn:\,N \omega(x)>\lambda\} = \bigcup_{t\in(0,\fz)}
\bigcup_{\gfz{y \in \rr^n}{\omega(y,t)>\lambda}} B(y,t)
\]
and that
\[
\{x\in\rn:\,N_\beta \omega(x)>\lambda\} = \bigcup_{t\in(0,\fz)}
\bigcup_{\gfz{y \in \rr^n}{\omega(y,t)>\lambda}} B(y,\beta t).
\]
By the Linder\"{o}f covering lemma, there exists a countable subset
$\{B_l\}_{l=0}^\fz$ of $\{B(y,t):\, t\in(0,\fz),\ y\in\rr^n\
\mathrm{satisfying}\ \omega(y,t)>\lambda\}$ such that
$\{x\in\rn:\,N_\beta\omega(x)>\lambda\}=\{\cup_{l=0}^\fz(\beta
B_l)\}$ and $ \{x\in\rn:\,N\omega(x)>\lambda\}\supset
(\cup_{l=0}^\fz B_l)$. By Lemma \ref{l3.1}, it suffices to prove
that
\[
H^d(\{x\in\rn:\, N_\beta \omega(x)>\lambda \},\{\beta
B_l\}_{l=0}^\fz) \ls \beta^d H^d\left(\bigcup_{l=0}^\fz B_l, \,
\{B_l\}_{l=0}^\fz\right).
\]
Let $\{B^*_k\}_{k=0}^\fz$ be a ball covering of $\cup_{l \in \nn}B_l
$ such that $\sum_{k=0}^\fz r_{B_k^*}^d\le 2H^d(\cup_{l=0}^\fz B_l,
\, \{B_l\}_{l=0}^\fz)$ and that $B_k^*$ engulfs $B_l$ whenever they
intersect, where $r_{B_k^*}$ denotes the radius of $B_k^*$.
Therefore, $\beta B_k^*$ engulfs $\beta B_l$ whenever they intersect
and $\{x\in\rn:\, N_\beta \omega(x)>\lambda \}\subset
\{\cup_{k=0}^\fz(\beta B_k^*)\}$. We then have
\[
2\beta^dH^d \left(\bigcup_{l=0}^\fz B_l, \, \{B_l\}_{l=0}^\fz\right)
\ge \sum_{l=0}^\fz(\beta r_{B_k^*})^d \ge H^d\l(\{x\in\rn:\,N_\beta
\omega(x)>\lambda \},\{\beta B_l\}_{l=0}^\fz\r),
\]
which completes the proof of Lemma \ref{l3.2}.
\end{proof}

As an immediate consequence of Lemma \ref{l3.2}, we have the
following result.

\begin{corollary}\label{c3.1}
Let $d\in(0,n]$, $\beta\in[1,\fz)$ and $\omega$ be a nonnegative
measurable function on $\rr^{n+1}_+$. Define $\omega_\beta(x,t)=
\sup_{y\in B(x,\beta t)}\omega(y,t)$. Then there exists a positive
constant $C$ such that
\[
\int_{\R^n}N\omega_\beta(x)\,dH^d(x) \le C\beta^d
\int_{\R^n}N\omega(x)\,dH^d(x).
\]
\end{corollary}

Now we turn to the proof of Theorem \ref{t3.1}.

\begin{proof}[Proof of Theorem \ref{t3.1}]
By similarity, we only consider $\dsfh$. Similarly to the proof of
\cite[Theorem 4.1]{yy2}, without loss of generality, we may assume
$s=0$, since this case implies the general case.

By the Aoki theorem (see \cite{ao}), there exists a $\kappa\in(0,1]$
such that $\|\cdot\|_{f\dot{H}^{0,\tau}_{p,q}(\rn)}^\kappa$ becomes
a norm in $f\dot{H}^{0,\tau}_{p,q}(\rn)$. Let $t\in f{\dot
H}_{p,q}^{0,\tau}(\rn)$. For $Q\in\cD(\rn)$, we write $A\equiv
A_0+A_1$ with $(A_0t)_Q\equiv \sum_{\{P\in\cD(\rn):\ \ell(Q)\le
\ell(P)\}}a_{QP}t_P$ and $(A_1t)_Q\equiv\sum_{\{P\in\cD(\rn):\
\ell(P)<\ell(Q)\}}a_{QP}t_P$. By Definition \ref{d3.1}, we see that
for $Q\in\cD(\rn)$,
$$|(A_0t)_Q|\ls \sum_{\{P\in\cD(\rn):\,\ell(Q)\le \ell(P)\}}
\l(\frac{\ell(Q)}{\ell(P)}\r)^{\frac{n+\varepsilon}{2}}
\frac{|t_P|}{\l(1+[\ell(P)]^{-1}|x_Q-x_P|\r)^{n+\varepsilon}}.$$ Thus,
we have
\begin{eqnarray*}
\|A_0t\|_{f\dot{H}_{p,q}^{0,\tau}(\rn)}&&\ls
\inf_\omega\l\|\l\{\sum_{j\in\zz}\sum_{Q\in\cD_j(\rn)}|Q|^{-\frac
q2}\chi_Q\l[\sum_{i=-\fz}^j\sum_{P\in\cD_i(\rn)}2^{(i-j)
\frac{n+\varepsilon}2}\r.\r.\r.\\
&&\hs\times\l.\l.\l.\frac{|t_P|[\omega(\cdot,2^{-j})]^{-1}}
{\l(1+2^i|x_Q-x_P|\r)^{n+\varepsilon}}\r]^q\r\}^{\frac
1q}\r\|_{L^p(\rn)}.
\end{eqnarray*}
Let $\omega$ be a nonnegative Borel measurable function satisfying
\eqref{1.1} and
\begin{equation*}
\left\|\left\{ \sum_{j\in\zz} \sum_{Q \in \cD_j(\rn)}|t_Q|^q
[\widetilde{\chi}_Q\omega(\cdot,2^{-j})]^{-q} \right\}^\frac1q
\right\|_{L^p(\rn)} \ls \| t \|_{f\dot{H}_{p,q}^{0,\tau}(\rn)}.
\end{equation*}
Let $A_{0,i}(Q)\equiv \{P\in\cD_i(\rn):\,2^i|x_P-x_Q|\le
\sqrt{n}/2\}$ and $A_{m,i}(Q)\equiv
\{P\in\cD_i(\rn):\,2^{m-1}\sqrt{n}/2<2^i|x_P-x_Q|\le2^m\sqrt{n}/2\}$
for all $i\in\zz$ and $m\in\zz_+$. Define $\omega_m(x, t)\equiv
2^{-mn\tau}\sup_{y\in B(x, \sqrt{n}2^{m+1}t)}\omega(y,t)$ for all
$(x,t)\in\rr^{n+1}_+$. Then $N\omega_m\ls2^{-mn\tau}
N_{\sqrt{n}2^{m+2}}\omega$ and $[\omega_m(x, 2^{-j})]^{-1}\omega(y,
2^{-i})\ls 2^{mn\tau}$ for $m\in\zz_+$, $x\in Q$ with
$Q\in\cD_j(\rn)$, $y\in P$ with $P\in A_{m,i}(Q)$ and $i\le j$.
Moreover, using Corollary \ref{c3.1}, we see that a constant
multiple of $\omega_m$ also satisfies \eqref{1.1}. Similarly to the
proof of Lemma \ref{l2.2}, we have that for all $x\in Q$,
\begin{eqnarray*}
&&\sum_{P\in A_{m,i}(Q)}\frac{|t_P|[\omega_m(x,2^{-j})]^{-1}}
{\l(1+2^i|x_Q-x_P|\r)^{n+\varepsilon}}\ls
2^{-m\varepsilon+mn\tau}\mathrm{HL}\left( \sum_{P \in
A_{m,i}(Q)}|t_P|\chi_P[\omega(\cdot,2^{-i})]^{-1} \right)(x).
\end{eqnarray*}
Hence, choosing $\varepsilon>n\tau$, by Fefferman-Stein's vector
valued inequality, we obtain
\begin{eqnarray*}
\|A_0t\|^\kappa_{f\dot{H}_{p,q}^{0,\tau}(\rn)}
&&\ls\sum_{m=0}^\fz\l\{\inf_\omega
\l\|\l\{\sum_{j\in\zz}\sum_{Q\in\cD_j(\rn)}|Q|^{-\frac
q2}\chi_Q\l[\sum_{i=-\fz}^j\sum_{P\in A_{m,i}(Q)}2^{(i-j)
\frac{n+\varepsilon}2}\r.\r.\r.\r.\\
&&\hs\times\l.\l.\l.\l.\frac{|t_P|[\omega(\cdot,2^{-j})]^{-1}}
{\l(1+2^i|x_Q-x_P|\r)^{n+\varepsilon}}\r]^q\r\}^{\frac
1q}\r\|_{L^p(\rn)}\r\}^\kappa\\
&&\ls\sum_{m=0}^\fz
\l\|\l\{\sum_{j\in\zz}\sum_{Q\in\cD_j(\rn)}|Q|^{-\frac
q2}\chi_Q\l[\sum_{i=-\fz}^j\sum_{P\in A_{m,i}(Q)}2^{(i-j)
\frac{n+\varepsilon}2}\r.\r.\r.\\
&&\hs\times\l.\l.\l.\frac{|t_P|[\omega_m(\cdot,2^{-j})]^{-1}}
{\l(1+2^i|x_Q-x_P|\r)^{n+\varepsilon}}\r]^q\r\}^{\frac
1q}\r\|^\kappa_{L^p(\rn)}\\
&&\ls \sum_{m=0}^\fz
2^{m(n\tau-\varepsilon)\kappa}\l\|\l\{\sum_{j\in\zz}\sum_{Q\in\cD_j(\rn)}\chi_Q
\l[\sum_{i=-\fz}^j2^{(i-j)\varepsilon/2}\r.\r.\r.\\
&&\hs\times\l.\l.\l.\mathrm{HL}\left( \sum_{P \in
A_{m,i}(Q)}|t_P|\wz\chi_P[\omega(\cdot,2^{-i})]^{-1}
\right)\r]^q\r\}^{\frac 1q}\r\|^\kappa_{L^p(\rn)}\ls
\|t\|^\kappa_{f\dot{H}_{p,q}^{0,\tau}(\rn)}.
\end{eqnarray*}

The proof for $A_1t$ is similar. Indeed, we have
$$|(A_1t)_Q|\ls \sum_{\{P\in\cD(\rn):\,\ell(P)\le \ell(Q)\}}
\l(\frac{\ell(P)}{\ell(Q)}\r)^{\frac{n+\varepsilon}{2}}
\frac{|t_P|}{\l(1+[\ell(Q)]^{-1}|x_Q-x_P|\r)^{n+\varepsilon}}.$$ Thus,
\begin{eqnarray*}
\|A_1t\|_{f\dot{H}_{p,q}^{0,\tau}(\rn)}&&\ls
\inf_\omega\l\|\l\{\sum_{j\in\zz}\sum_{Q\in\cD_j(\rn)}|Q|^{-\frac
q2}\chi_Q\l[\sum_{l=0}^\fz\sum_{P\in\cD_{j+l}(\rn)}2^{-l
\frac{n+\varepsilon}2}\r.\r.\r.\\
&&\hs\times\l.\l.\l.\frac{|t_P|[\omega(\cdot,2^{-j})]^{-1}}
{\l(1+2^j|x_Q-x_P|\r)^{n+\varepsilon}}\r]^q\r\}^{\frac
1q}\r\|_{L^p(\rn)}.
\end{eqnarray*}
Let $\wz{A}_{0,j,l}(Q)\equiv \{P\in\cD_{j+l}(\rn):\,2^j|x_P-x_Q|\le
\sqrt{n}/2\}$ and $\wz{A}_{m,j,l}(Q)\equiv
\{P\in\cD_{j+l}(\rn):\,2^{m-1}\sqrt{n}/2<2^j|x_P-x_Q|\le2^m\sqrt{n}/2\}$
for all $j\in\zz$ and $m,\,l\in\zz_+$. Set
$$
\wz\omega_m(x,s)\equiv2^{-(m+l)n\tau} \sup\{ \omega(y,s) \, : \,
y\in\rn,\ |y-x|<\sqrt{n}2^{m+l+1}s\}
$$for all $m \in \zz_+$ and
$(x,s)\in\rr^{n+1}_+$. Similarly, we have that a constant multiple
of $\wz\omega_m$ satisfies \eqref{1.1} and $[\wz\omega_m(x,
2^{-j})]^{-1}\omega(y, 2^{-j-l})\ls 2^{(m+l)n\tau}$ for
$m,\,l\in\zz_+$, $x\in Q$ with $Q\in\cD_j(\rn)$, $y\in P$ with $P\in
\wz{A}_{m,j,l}(Q)$. Similarly to the
proof of Lemma \ref{l2.2} again, we see that for all $x\in Q$,
\begin{eqnarray*}
&&\sum_{P\in\wz{A}_{m,j,l}(Q)}
\frac{|t_P|[\wz\omega_m(x,2^{-j})]^{-1}}{(1+2^j|x_Q-x_P|)^{n+\varepsilon}}\ls
2^{-m\varepsilon+ln+(m+l)n\tau}\mathrm{HL}\left( \sum_{P \in
\wz{A}_{m,j,l}(Q)}\frac{|t_P|\chi_P}{\omega(\cdot,2^{-i})}
\right)(x).
\end{eqnarray*}
Hence, choosing $\varepsilon>2n\tau$, similarly to the estimate of
$\|A_0t\|_{f\dot{H}_{p,q}^{0,\tau}(\rn)}$, we also have
\begin{eqnarray*}
\|A_1t\|^\kappa_{f\dot{H}_{p,q}^{0,\tau}(\rn)} &&\ls\sum_{m=0}^\fz
\l\|\l\{\sum_{j\in\zz}\sum_{Q\in\cD_j(\rn)}|Q|^{-\frac
q2}\chi_Q\l[\sum_{l=0}^\fz\sum_{P\in \wz{A}_{m,j,i}(Q)}2^{-l
\frac{n+\varepsilon}2}\r.\r.\r.\\
&&\hs\times\l.\l.\l.\frac{|t_P|[\omega_m(\cdot,2^{-j})]^{-1}}
{\l(1+2^j|x_Q-x_P|\r)^{n+\varepsilon}}\r]^q\r\}^{\frac
1q}\r\|^\kappa_{L^p(\rn)}\\
&&\ls \sum_{m=0}^\fz
2^{m(n\tau-\varepsilon)\kappa}\l\|\l\{\sum_{j\in\zz}\sum_{Q\in\cD_j(\rn)}\chi_Q
\l[\sum_{l=0}^\fz 2^{-l(\varepsilon/2-n\tau)}\r.\r.\r.\\
&&\hs\times\l.\l.\l.\mathrm{HL}\left( \sum_{P \in
\wz{A}_{m,j,i}(Q)}|t_P|\wz\chi_P[\omega(\cdot,2^{-i})]^{-1}
\right)\r]^q\r\}^{\frac 1q}\r\|^\kappa_{L^p(\rn)}\ls
\|t\|^\kappa_{f\dot{H}_{p,q}^{0,\tau}(\rn)},
\end{eqnarray*}
which completes the proof of Theorem \ref{t3.1}.
\end{proof}

\begin{remark}\label{r3.1} We point out that Theorem \ref{t3.1} generalizes the corresponding
results of Besov Spaces and Triebel-Lizorkin spaces in \cite{fj88,fj,fjw}
when $p\in(1,\fz)$ and $q\in[1,\fz)$ by taking $\tau=0$.
Moreover, the restriction that $\epsilon>2n\tau$ in Theorem \ref{t3.1} is
different from the restriction that $\epsilon>2n(\tau-1/p)$
in \cite[Theorem 4.1]{yy2} on the spaces $\dbt$ and
$\dft$.
\end{remark}

As applications of Theorem \ref{t3.1}, we establish the smooth
atomic and molecular decomposition characterizations of $\dah$.

\begin{definition}\label{d3.2}
Let $p\in(1,\fz), \, q\in[1,\fz), \, s \in \R$, $\tau\in[0,
\frac{1}{(p \vee q)'}]$ and $Q\in\cD(\rn)$. Set $N \equiv
\max(\lfloor -s+2n\tau \rfloor,-1)$ and $s^*\equiv s-\lfloor s
\rfloor$.

(i) A function $m_Q$ is called a \textit{smooth synthesis molecule
for $\dah$ supported near $Q$}, if there exist a $\delta \in
(\max\{s^*,(s+n\tau)^*\},1]$ and $M>n+2n\tau$ such that
$\int_{\rn}x^\gamma m_Q(x)\,dx=0$ if $|\gamma| \le N$, $|m_Q(x)|
\le|Q|^{-\frac12}\left(1+[\ell(Q)]^{-1}|x-x_Q|\right)^{-\max(M,M-s)}$,
\begin{equation}\label{3.1}
|\partial^\gamma m_Q(x)| \le |Q|^{-\frac12-\frac{|\gamma|}{n}}
\left(1+[\ell(Q)]^{-1}|x-x_Q|\right)^{-M} \ {\rm if}\ |\gamma| \le
\lfloor s+3n\tau \rfloor,
\end{equation}
and
\begin{eqnarray}\label{3.2}
&&|\partial^\gamma m_Q(x)-\partial^\gamma m_Q(y)|\nonumber\\
&&\hs\le |Q|^{-\frac12-\frac{|\gamma|}{n}-\frac{\delta}{n}}
|x-y|^\delta \sup_{|z| \le |x-y|}
(1+[\ell(Q)]^{-1}|x-z-x_Q|)^{-M}
\end{eqnarray} \
if $|\gamma| = \lfloor s+3n\tau \rfloor$.

A set $\{m_Q\}_{Q\in\cD(\rn)}$ of functions is called a family of
smooth synthesis molecules for $\dah$, if each $m_Q$ is a smooth
synthesis molecule for $\dah$ supported near $Q$.

(ii) A function $b_Q$ is called a \textit{smooth analysis molecule
for $\dah$ supported near $Q$}, if there exist a $\rho \in
((n-s)^*,1]$ and $M>n+2n\tau$ such that $ \int_{\rn}x^\gamma
b_Q(x)\,dx=0$ if $|\gamma| \le \lfloor s+3n\tau \rfloor$, $
|b_Q(x)|\le|Q|^{-\frac12}\left(1+[\ell(Q)]^{-1}
|x-x_Q|\right)^{-\max(M,M+s+n\tau)},$
\begin{equation}\label{3.3}
|\partial^\gamma b_Q(x)| \le |Q|^{-\frac12-\frac{|\gamma|}{n}}
\left(1+[\ell(Q)]^{-1}|x-x_Q|\right)^{-M} \ {\rm if}\ |\gamma| \le N,
\end{equation}
and
\begin{eqnarray}\label{3.4}
&&|\partial^\gamma b_Q(x)-\partial^\gamma b_Q(y)|\nonumber\\
&&\hs\le
|Q|^{-\frac12-\frac{|\gamma|}{n}-\frac{\delta}{n}} |x-y|^\delta
\sup_{|z| \le |x-y|} (1+[\ell(Q)]^{-1}|x-z-x_Q|)^{-M} \ {\rm if}\
|\gamma| = N.
\end{eqnarray}

A set $\{b_Q\}_{Q\in\cD(\rn)}$ of functions is called a family of
smooth analysis molecules for $\dah$, if each $b_Q$ is a smooth
analysis molecule for $\dah$ supported near $Q$.
\end{definition}

We remark that if $s+3n\tau<0$, then \eqref{3.1} and \eqref{3.2} are
void; if $N<0$, then \eqref{3.3} and \eqref{3.4} are void. By a
similar argument to the proof of \cite[Corollary B.3]{fj} (see also
\cite[Lemma 4.1]{yy2}), we have the following conclusion.

\begin{lemma}
\label{l3.3} Let $p\in(1,\fz), \, q\in[1,\fz), \, s \in \R$ and
$\tau\in [0,\frac1{(p\vee q)'}]$. Then there exist
$\varepsilon_1>2n\tau$ and a positive constant $C$ such that for all
families $\{m_Q\}_{Q \in \cD(\rn)}$ of smooth synthesis molecules
for $\dah$ and families $\{b_Q\}_{Q \in \cD(\rn)}$ of smooth
analysis molecules for $\dah$, $ |\langle m_P,\,b_Q
\rangle_{L^2(\rn)}| \le C\,\omega_{QP}(\varepsilon_1). $
\end{lemma}

To formulate the molecular decomposition,
the following lemma is indispensable.

\begin{lemma}
\label{l3.4} Retain the same assumptions as in Lemma \ref{l3.3}. Let
$f\in\dah$ and $\Phi$ be a smooth analysis molecule for $\dah$
supported near a dyadic cube $Q$. Then $\langle f,\,\Phi\rangle$ is
well defined. Indeed, let $\varphi,\psi \in \cS(\rn)$ be as in
\eqref{2.2}. Then the series
\begin{equation}\label{3.5}
\langle f,\,\Phi\rangle\equiv\sum_{j\in\zz}\langle
\widetilde{\varphi}_j\ast\psi_j\ast f, \Phi\rangle=\sum_{P \in
\cD(\rn)}\langle f,\varphi_P \rangle \langle \psi_P,\Phi \rangle
\end{equation}
converges absolutely and its value is independent of the choices
of $\varphi$ and $\psi$.
\end{lemma}

\begin{proof}
The same proof as that of \cite[Lemma 4.2]{yy2} works for the
absolute convergence of \eqref{3.5}. We only need to prove that the
value of \eqref{3.5} is independent of the choices of $\varphi$ and
$\psi$. By similarity again, we only consider the spaces $\dbh$.

Let $f\in\dbh$. We claim that $\sum_{j=0}^\infty
\widetilde{\varphi}_j\ast\psi_j\ast f$ converges in $\cS'(\rn)$. In
fact, similarly to the proof of \cite[Lemma 2.2]{yy1}, we have that
for all $\phi\in\cS(\rn)$ and $x\in\rn$, $$| \vz_j\ast\phi(x) |
\lesssim \| \vz \|_{\cS_{M+1}} \| \phi \|_{\cS_{M+1}}
\frac{2^{-jM}}{(1+|x|)^{n+M}}, $$
where $M\in\nn$ is determined later. Thus,
$$\sum_{j=0}^\infty |\langle \widetilde{\varphi}_j\ast\psi_j\ast
f,\phi \rangle| \lesssim \| \vz \|_{\cS_{M+1}}\| \phi \|_{\cS_{M+1}}
\sum_{j=0}^\infty 2^{-jM} \int_{\R^n} \frac{|\psi_j\ast
f(x)|}{(1+|x|)^{n+M}}\,dx.$$
Recall again that $\omega(x,t) \ls t^{-n\tau}$ for all nonnegative
Borel measurable functions $\omega$ on $\rr^{n+1}_+$ satisfying
\eqref{1.1}. Letting $M>\max(0, n\tau-s)$, by H\"older's inequality,
we then obtain
\begin{eqnarray*}
\sum_{j=0}^\infty |\langle \widetilde{\varphi}_j\ast\psi_j\ast
f,\phi \rangle| &&\lesssim \| \vz \|_{\cS_{M+1}} \| \phi
\|_{\cS_{M+1}} \sum_{j=0}^\infty 2^{-jM+jn\tau} \int_{\R^n}
\frac{|\psi_j\ast f(x)|[\omega(x,2^{-j})]^{-1}}
{(1+|x|)^{n+M}}\,dx\\
&&\lesssim \| \vz \|_{\cS_{M+1}}\| \phi \|_{\cS_{M+1}} \| f
\|_{\dbh},
\end{eqnarray*}
which implies that $\sum_{j=0}^\infty
\widetilde{\varphi}_j\ast\psi_j\ast f$ converges in $\cS'(\rn)$.
Thus, the claim is true.

We need to handle carefully the remaining summation:
$\sum_{j=-\infty}^{-1} \widetilde{\varphi}_j\ast\psi_j\ast f$. In
general it is not possible to prove that $\sum_{j=-\infty}^{-1}
\widetilde{\varphi}_j\ast\psi_j\ast f$ is convergent in $\cS'(\rn)$.
Therefore, we pass to its partial derivatives. Choose
$\gamma\in\zz_+^n$ such that $|\gamma|>s-n\tau-n/p$. Then using
H\"older's inequality, similarly to the previous estimate, we obtain
that for all $x\in\rn$,
\begin{eqnarray*}
\sum_{j=-\infty}^{-1}
|\partial^\gamma(\widetilde{\varphi}_j\ast\psi_j\ast
f)(x)|&&\ls\sum_{j=-\infty}^{-1}2^{j(n+|\gz|)}\|\vz\|_{\cS_{M+1}}\int_\rn\frac{|\psi_j\ast
f(y)|}{(1+2^j|x-y|)^{n+M+|\gz|}}\,dy\\
&&\lesssim
\sum_{j=-\infty}^{-1}2^{j\left(|\gamma|-s+n\tau+\frac{n}{p}\right)}
\| \varphi \|_{\cS_{M+1}} \| f
\|_{\dbh}\\
&&\ls \| \varphi \|_{\cS_{M+1}} \| f \|_{\dbh}.
\end{eqnarray*}
Therefore, it follows from the well-known result in \cite[Remark
B.4]{fj} or \cite[Lemma 5.4]{bh} that there exist a sequence
$\{P_N\}_{N \in \N}$ of polynomials on $\rn$ with degree no more
than $ \max\left(-1,\lfloor s-n\tau-n/p\rfloor\right)$ and $g \in
\cS'(\rn)$ such that $g=\lim_{N\to\fz}(\sum_{j=-N}^\infty
\widetilde{\varphi}_j\ast\psi_j\ast f+P_N )$ in $\cS'(\rn)$ and $g$
is a representative of the equivalence class $f+\cP(\rn)$; see
\cite[pp.\,153-154]{fj}. Using \cite[Lemma 5.4]{bh} and repeating
the argument in \cite[pp.\,153-154]{fj}, we obtain that the value of
\eqref{3.5} is independent of the choices of $\varphi$ and $\psi$,
which completes the proof of Lemma \ref{l3.4}.
\end{proof}
%

With Theorem \ref{t3.1}, Lemmas \ref{l3.3} and \ref{l3.4}, we now
have the following smooth molecular decomposition of $\dah$. The
proof of Theorem \ref{t3.2} parallels the proofs of \cite[Theorem
4.2]{yy2} and \cite[Theorems 3.5, 3.7]{fj}. We omit the details.

\begin{theorem}\label{t3.2}
Let $s,\,p,\,q$ and $\tau$ be as in Lemma \ref{l3.3}.

(i) If $\{m_Q\}_{Q \in \cD(\rn)}$
is a family of smooth synthesis molecules for $\dah$,
then there exists a positive constant $C$
such that for all $t=\{t_Q\}_{Q \in \cD(\rn)} \in\dsah$,
$$\l\|\sum_{Q \in \cD(\rn)}t_Qm_Q \r\|_{\dah}
\le C \| t \|_{\dsah}.$$

(ii) If $\{b_Q\}_{Q \in \cD(\rn)}$ is a family of smooth analysis
molecules for $\dah$,
then there exist a positive constant $C$
such that for all $f \in \dah$,
$$\| \{ \langle f,b_Q\rangle\}_{Q \in \cD(\rn)} \|_{\dsah}
\le C \| f \|_{\dah}.$$
\end{theorem}

Theorem \ref{t3.2} generalizes the well known results on
$\dot{B}^s_{p,q}(\rn)$ and $\dot{F}^s_{p,q}(\rn)$ in
\cite{fj85,fj88,fj,fjw,b05,bh} by taking $\tau=0$.

Next we establish the smooth atomic decomposition characterizations
of $\dah$.

\begin{definition}\label{d3.3}
Let $s \in \R$, $p\in(1,\infty)$, $q\in[1,\infty)$, $\tau$ and $N$
be as in Definition \ref{d3.2}. A function $a_Q$ is called \textit{a
smooth atom for $\dah$ supported near a dyadic cube $Q$}, if there
exist $\widetilde{K}$ and $\widetilde{N}$ with $\widetilde{K} \ge
\max(\lfloor s+3n\tau+1 \rfloor, 0)$ and $\widetilde{N}\ge N$ such
that $a_Q$ satisfies the following support, regularity and moment
conditions: $\supp a_Q \subset 3Q$, $\|\partial^\gamma
a_Q\|_{L^\infty(\rn)} \le |Q|^{-\frac12-\frac{|\gamma|}{n}}$ if
$|\gamma| \le \widetilde{K}$, and $\int_{\R^n}x^\gamma a_Q(x)\,dx=0$
if $|\gamma| \le \widetilde{N}.$

A set $\{a_Q\}_{Q\in\cD(\rn)}$ of functions is called a family of smooth
atoms for $\dah$, if each $a_Q$ is a smooth atom for $\dah$
supported near $Q$.
\end{definition}

\begin{remark}\label{r3.2} We point out that in Definition \ref{d3.3},
the regularity condition of smooth atoms can be strengthened into
that $\|\partial^\gamma a_Q\|_{L^\infty(\rn)} \le
|Q|^{-\frac12-\frac{|\gamma|}{n}}$ for all $|\gz|\le M$, where $M$
can be any sufficiently large constant depending on $s,\,\tau,\,p$
and $q$; see Grafakos \cite[Definition 6.6.2]{g08} for the details.
\end{remark}

It is clear that every smooth atom for $\dah$ is a constant multiple
of a smooth synthesis molecule $\dah$. Once we establish Theorem
\ref{t3.2}, an argument used in \cite[pp.\,60-61]{fj} or
\cite[pp.\,1495-1497]{bh} yields the following conclusion; we omit
the details.

\begin{theorem}\label{t3.3}
Let $s,\,p,\,q,\,\tau$ be as in Lemma \ref{l3.3}. Then for each $f
\in \dah$, there exist a family $\{a_Q\}_{Q \in \cD(\rn)}$ of smooth
atoms for $\dah$, a coefficient sequence $t\equiv\{t_Q\}_{Q \in
\cD(\rn)}\in\dsah$, and a positive constant $C$ such that $f=\sum_{Q
\in \cD(\rn)}t_Qa_Q$ in $\cS'_\infty(\rn)$ and $\| t \|_{\dsah}\le C
\| f \|_{\dah}$.

Conversely, there exists a positive constant $C$ such that for all
families $\{a_Q\}_{Q \in \cD(\rn)}$ of smooth atoms for $\dah$ and
coefficient sequences $t\equiv\{t_Q\}_{Q \in \cD(\rn)} \in \dsah$,
$\|\sum_{Q \in \cD(\rn)}t_Qa_Q \|_{\dah} \le C \|t \|_{\dsah}$.
\end{theorem}

Theorem \ref{t3.3} again generalizes the well known results on
$\dot{B}^s_{p,q}(\rn)$ and $\dot{F}^s_{p,q}(\rn)$ in \cite{fj85,fj88,fj,fjw}
(see also \cite{b05,bh,g08}) by taking $\tau=0$.

\section{Pseudo-differential operators and trace theorems}

\hskip\parindent In this section, we give some applications of the
smooth atomic and molecular decomposition characterizations of
$\dah$, including the boundedness of pseudo-differential
operators with homogeneous symbols in these spaces and their trace
properties. We first recall the notion of homogeneous symbols; see,
for example, \cite{gt}.

\begin{definition}\label{d4.1}
Let $m \in \Z$. A smooth function $a$ defined on $\R^n_x \times
(\R^n_\xi \setminus \{0\})$ belongs to the class
$\dot{S}_{1,1}^m(\rn)$, if $a$ satisfies the following differential
inequalities that for all $\alpha,\beta \in \Z^n_+$,
\[
\sup_{x \in \R^n, \, \xi \in (\R^n \setminus \{0\})}
|\xi|^{-m-|\alpha|+|\beta|}|\partial_x^\alpha
\partial_\xi^\beta a(x,\xi)|<\fz.
\]
\end{definition}

As an application of the smooth molecular decomposition of $\dah$
(Theorem \ref{t3.2}) and the Calder\'on reproducing formula
\eqref{2.2}, we have the following conclusion.

\begin{theorem}\label{t4.1}
Let $m \in \Z$, $s\in\R$, $p\in(1,\infty)$, $q\in[1,\infty)$ and
$\tau\in[0, \frac{1}{(p \vee q)'}]$. Let $a$ be a symbol in
$\dot{S}_{1,1}^m(\rn)$ and $a(x,D)$ be the pseudo-differential
operator such that
\[
a(x,D)f(x) \equiv \int_{\R^n}a(x,\xi)(\cF f)(\xi)e^{ix\xi}\,d\xi
\]
for all smooth synthesis molecules for
$A\dot{H}_{p,q}^{s+m,\tau}(\rn)$ and $x \in \R^n$. Assume that its
formal adjoint $a(x,D)^*$ satisfies $ a(x,D)^*(x^\beta)=0$ in
$\cS'_\infty(\rn)$ for all $\beta \in \Z^n_+$ with $|\beta| \le
\max\{-s+2n\tau,-1\}$. Then $a(x,D)$ is a bounded linear operator
from $A\dot{H}_{p,q}^{s+m,\tau}(\rn)$ to $\dah$.
\end{theorem}

\begin{proof}
The proof is similar to that in \cite{ftw,fhjw,t90,t91,gt}; see also \cite{syy1}. We
abbreviate $T \equiv a(x,D)$ for simplicity. Let $f \in
A\dot{H}_{p,q}^{s+m,\tau}(\rn)$ and $\varphi$ be as in Definition
\ref{d1.1} such that for all $\xi \in \R^n$, $\sum_{j \in
\Z}|\cF\varphi(2^{-j}\xi)|^2 = \chi_{\R^n \setminus \{0\}}(\xi)$.
Then by the Calder\'on reproducing formula \eqref{2.2}, we have $f
\equiv \sum_{Q \in \cD(\rn)} \langle f,\varphi_Q \rangle \varphi_Q $
in $\cS'_\infty(\rn)$; moreover, by the $\varphi$-transform
characterization of $A\dot{H}_{p,q}^{s+m,\tau}(\rn)$ (see Theorem
\ref{t2.1}), we see that $\| \{\langle f,\varphi_Q \rangle\}_{Q \in
\cD(\rn)} \|_{a\dot{H}_{p,q}^{s+m,\tau}(\rn)} \lesssim \| f
\|_{A\dot{H}_{p,q}^{s+m,\tau}(\rn)}$, or equivalently,
$\|\{|Q|^{-\frac{m}{n}}\langle f,\varphi_Q \rangle\}_{Q \in
\cD(\rn)}\|_{\dsah} \lesssim \| f
\|_{A\dot{H}_{p,q}^{s+m,\tau}(\rn)}$.

We claim that $T(f)\equiv\sum_{Q \in \cD(\rn)} \langle f,\varphi_Q
\rangle T(\varphi_Q) $ in $\cS'_\infty(\rn)$ with
$\|T(f)\|_{\dah}\ls\|f\|_{A\dot{H}_{p,q}^{s+m,\tau}(\rn)}$. To this
end, by Theorem \ref{t3.2} (i), it suffices to show that every
$|Q|^{\frac{m}{n}}T(\varphi_Q)$ is a constant multiple of a
synthesis molecule for $\dah$ supported near $Q$. This fact was
established by Grafakos and Torres \cite{gt}; see also \cite{syy1}.
We then conclude that $T$ is bounded from
$A\dot{H}_{p,q}^{s+m,\tau}(\rn)$ to $\dah$, which completes the proof of
Theorem \ref{t4.1}.
\end{proof}

We remark that Theorem \ref{t4.1} generalizes the corresponding
classical results in Besov spaces and Triebel-Lizorkin spaces
obtained by Grafakos and Torres \cite[Theorems 1.1 and 1.2]{gt} when
$p\in(1,\fz)$ and $q\in[1,\fz)$ by taking $\tau=0$.

As an application of smooth atomic decomposition of $\dah$, we are
now going to show the trace theorem. For $x=(x_1,\cdots,x_n)\in\rn$,
we set $x'=(x_1,\cdots,x_{n-1})\in\R^{n-1}$.

\begin{theorem}\label{t4.2}
Let $n\ge2$, $p\in(1,\infty)$, $q\in[1,\infty)$, $\tau\in[0,
\frac{n-1}{n(p\vee q)'}]$ and $s\in(\frac1p+2n\tau,\fz)$. Then there
exists a surjective and continuous operator
$${\rm Tr}: f\in\dah \mapsto {\rm Tr}(f)\in
A\dot{H}_{p,q}^{s-\frac1p,\frac{n}{n-1}\tau}(\R^{n-1})$$ such that
${\rm Tr}(f)(x')=f(x',0)$ for all $x'\in\R^{n-1}$ and smooth atoms
$f$ for $\dah$.
\end{theorem}

To prove this theorem, we need the following technical lemma.

\begin{lemma}\label{l4.1}
Let $d\in (0, n]$ and $\Omega$ be an open set in $\rn$. Define
\[
H^d_*(\Omega) \equiv \inf \left\{ \sum_{j=1}^\infty r_j^d \, : \,
\Omega \subset \bigcup_{j=1}^\infty B(x_r,r_j), \, r_j>\frac{{\rm
dist}(x_j,\partial \Omega)}{10000} \right\}.
\]
Then $H^d(\Omega)$ and $H^d_*(\Omega)$ are equivalent for all
$\Omega$.
\end{lemma}

\begin{proof}
The inequality $H^d(\Omega)\le H^d_*(\Omega)$ is trivial from the
definitions. To prove the converse, we choose a ball covering
$\{B(x_j,r_j)\}_{j=1}^\infty$ of $\Omega$ such that
$\sum_{j=1}^\infty r_j^d \le 2H^d(\Omega)$. Let
$\{B(X_j,R_j)\}_{j=1}^\infty$ be a Whitney covering of $\Omega$
satisfying $\Omega=\cup_{j=1}^\infty B(X_j,R_j)$, $R_j/1000\le {\rm
dist}(X_j,\partial \Omega) \le R_j/100$ and $\sum_{j \in
\N}\chi_{R_j} \le C_n$; see, for example, \cite[Proposition
7.3.4]{g08}. Set
\[
J_1\equiv\left\{ j \in \N \, : \, (B(X_j,R_j) \cap
B(x_k,r_k))\neq\emptyset\ \mbox{and}\ R_j\le 4r_k\ \mbox{for some $k
\in \N$ }\right\}
\]
and $J_2\equiv (\N \setminus J_1)$. Notice that if $k \in \N$
satisfies $(B(X_j,R_j) \cap B(x_k,r_k)) \ne \emptyset$ for some $j
\in J_2$, then $B(x_k,r_k) \subset B(X_j,2R_j)$, since $r_k <R_j/4$.
With this in mind, we define
\[
K_2\equiv\{ k \in \N \, : \, (B(x_k,r_k)\cap B(X_j,R_j))\ne
\emptyset\
 \mbox{ for some } j \in J_2 \},
\]
and $K_1\equiv(\N \setminus K_2)$. It is easy to see that
\begin{equation}\label{4.1}
\bigcup_{k=1}^\infty B(x_k,r_k) \subset \l(\bigcup_{k \in
K_1}B(x_k,r_k) \bigcup \bigcup_{j \in J_2}B(X_j,2R_j)\r).
\end{equation}
Furthermore, for each $k \in \N$, the cardinality of the set $
\{j\in J_2:\, (B(x_k,r_k) \cap B(X_j,R_j)) \ne \emptyset\}$ is
bounded by a constant depending only on the dimension. Hence, we
have
\begin{align*}
\sum_{k=1}^\infty r_k^d &= \sum_{k \in K_1}r_k^d + \sum_{k \in
K_2}r_k^d\sim \sum_{k \in K_1}r_k^d + \sum_{j \in J_2} \left(
\sum_{k \in K_2, \, (B(x_k,r_k) \cap B(X_j,R_j)) \ne \emptyset}
r_k^d
\right)\\
&\sim \sum_{k \in K_1}r_k^d + \sum_{j \in J_2} \left( \sum_{k \in
K_2, \, (B(x_k,r_k) \cap B(X_j,R_j)) \ne \emptyset}
|B(x_k,r_k)|^{\frac{d}{n}} \right).
\end{align*}
Notice that $B(X_j,R_j) \subset \Omega \subset (\cup_{k=1}^\infty
B(x_k,r_k))$. Then for each $j \in J_2$, we have
\[
B(X_j,R_j) \subset \l\{\bigcup_{k \in K_2, \, (B(x_k,r_k) \cap
B(X_j,R_j)) \ne \emptyset}B(x_k,r_k)\r\}
\]
Since $d \in(0, n]$, by the monotonicity of $l^{\frac dn}$, we see
that
\begin{eqnarray*}
&&\left( \sum_{k \in K_2, \, (B(x_k,r_k) \cap B(X_j,R_j)) \ne
\emptyset}
|B(x_k,r_k)|^{\frac{d}{n}} \right)\\
&&\hs\ge \left( \sum_{k \in K_2, \, (B(x_k,r_k) \cap B(X_j,R_j)) \ne
\emptyset} |B(x_k,r_k)|\right)^{\frac{d}{n}}\ge
|B(X_j,R_j)|^{\frac{d}{n}}.
\end{eqnarray*}
As a consequence, $\sum_{k=0}^\fz r_k^d \gtrsim \sum_{k \in
K_1}r_k^d + \sum_{j \in J_2}R_j^d$, which combined with \eqref{4.1}
yields that $H^d_*(\Omega) \le \sum_{k \in K_1}r_k^d + \sum_{j \in
J_2}(2R_j)^d\ls \sum_{k \in K_1}r_k^d + \sum_{j \in J_2}R_j^d
\lesssim \sum_{k=0}^\fz r_k^d \lesssim H^d(\Omega)$. This finishes
the proof of Lemma \ref{l4.1}.
\end{proof}

\begin{proof}[Proof of Theorem \ref{t4.2}]
For similarity, we concentrate on the space $\dbh$. By Theorem
\ref{t3.3}, any $f \in \dbh$ admits a smooth atomic decomposition
$f=\sum_{Q \in \cD(\rn)}t_Q a_Q$ in $\cS'_\fz(\rn)$, where each
$a_Q$ is a smooth atom for $\dbh$ and
$t\equiv\{t_Q\}_{Q\in\cD(\rn)}\subset\C$ satisfies
$\|t\|_{\dsbh}\lesssim\|f \|_{\dbh}. $ Since $s>1/p+2n\tau$, there
is no need to postulate any moment condition on $a_Q$. Define
\[
{\rm Tr}(f)(*')\equiv\sum_{Q \in \cD(\rn)}t_Q a_Q(*',0) =\sum_{Q \in
\cD(\rn)}\frac{t_Q}{[\ell(Q)]^\frac12} [\ell(Q)]^{\frac12}a_Q(*',0).
\]
By the support condition of smooth atoms, the above
summation can be re-written as
\begin{equation}\label{4.2}
{\rm Tr}(f)(*')\equiv \sum_{i=0}^2\sum_{Q' \in
\cD(\R^{n-1})}\frac{t_{Q'\times[(i-1)\ell(Q'),
i\ell(Q'))}}{[\ell(Q')]^\frac12}
[\ell(Q')]^{\frac12}a_{Q'\times[(i-1)\ell(Q'), i\ell(Q'))}(*',0).
\end{equation}

We need to show that \eqref{4.2} converges in $\cS'_\fz(\R^{n-1})$
and
$$\|{\rm
Tr}(f)\|_{B\dot{H}_{p,q}^{s-\frac1p,\frac{n}{n-1}\tau}(\R^{n-1})}\ls
\|f\|_{\dbh}.$$ To this end, by Theorem \ref{t3.3}, it suffices to
prove that each $[\ell(Q')]^\frac12a_{Q' \times
[(i-1)\ell(Q'),i\ell(Q'))}(*',0)$ is a smooth atom for
$B\dot{H}_{p,q}^{s-\frac{1}{p},\frac{n}{n-1}\tau}(\R^{n-1})$
supported near $Q'$ and for all $i\in\{0,\,1,\,2\}$,
\begin{equation}\label{4.3}
\l\|\l\{[\ell(Q')]^{-\frac12}t_{Q' \times
[(i-1)\ell(Q'),i\ell(Q'))}\r\}_{Q'\in
\cD(\R^{n-1})}\r\|_{b\dot{H}_{p,q}^{s-\frac{1}{p},
\frac{n}{n-1}\tau}(\R^{n-1})}<\fz.
\end{equation} Indeed, it was
already proved in \cite{syy1} that $[\ell(Q')]^\frac12a_{Q' \times
[(i-1)\ell(Q'),i\ell(Q'))}(*',0)$ is a smooth atom for
$B\dot{H}_{p,q}^{s-\frac{1}{p},\frac{n}{n-1}\tau}(\R^{n-1})$. By
similarity, we only prove \eqref{4.3} when $i=1$. Let $\omega$ be a
nonnegative function on $\R^{n+1}_+$ satisfying \eqref{1.1} and
$$\left\{ \sum_{j \in \Z} \left[ \sum_{Q \in
\cD_j(\rn)}|Q|^{-(\frac{s}{n}+\frac12)p}|t_Q|^p\int_Q[\omega(x,2^{-j})]^{-p}\,dx
\right]^{\frac qp}\right\}^\frac1q\ls \|t\|_{\dsbh}.$$ For all
$\lambda\in(0,\fz)$, set $E_\lambda\equiv
\{x\in\rn:\,[N\omega(x)]^{(p\vee q)'}>\lambda\}$. Then there exists
a ball covering $\{B_m\}_m$ of $E_\lambda$ such that
\begin{equation}\label{4.4}
H^{n\tau (p\vee q)'}(E_\lambda) \sim \sum_m r_{B_m}^{n\tau (p\vee
q)'},
\end{equation}
where $r_{B_m}$ denotes the radius of $B_m$. Let
$\wz{H}^{n\tau(p\vee q)'}$ be the $(n-1)\frac {n\tau}{n-1}(p\vee
q)'$-Hausdorff capacity in $\R^{n-1}$ and define $\wz\omega$ on
$\R^n_+$ by setting, for all $x'\in\R^{n-1}$ and $t\in(0,\fz)$,
$\wz\omega(x',t)\equiv \wz{C}\sup_{\{x_n \in {\mathbb R} \, : \,
|x_n|<t\}}\omega((x',x_n),t)$, where $\wz{C}$ is a positive constant
chosen so that $N\wz\omega(x')\le N\omega(x',0)$ for all
$x'\in\R^{n-1}$. Therefore, if $[N\wz\omega(x')]^{(p\vee
q)'}>\lambda$, then $[N\omega(x',0)]^{(p\vee q)'}>\lambda$, and
hence $(x',0)\in B_m$ for some $m$, which further implies that
$\wz{E}_\lambda\equiv\{x'\in\R^{n-1}:\,[N\wz\omega(x')]^{(p\vee
q)'}>\lambda\}\subset (\cup_m B^*_m)$, where $B^*_m$ is the
projection of $B_m$ from $\rn$ to $\R^{n-1}$. This combined with
\eqref{4.4} further yields that $$\int_{\R^{n-1}}
[N\wz\omega(x')]^{(p \vee q)'}d\wz{H}^{n\tau(p \vee
q)'}(x')=\int_0^\fz \wz{H}^{n\tau(p \vee
q)'}(\wz{E}_\lambda)\,d\lambda\ls\int_0^\fz H^{n\tau(p \vee
q)'}(E_\lambda)\,d\lambda \ls 1.$$ Furthermore,
\begin{eqnarray*}
&&\l\|\l\{[\ell(Q')]^{-\frac12}t_{Q' \times [0,\ell(Q'))}\r\}_{Q'\in
\cD(\R^{n-1})}\r\|_{b\dot{H}_{p,q}^{s-\frac{1}{p},
\frac{n}{n-1}\tau}(\R^{n-1})}\\
&&\hs\ls \left\{ \sum_{j \in \Z} \left[ \sum_{Q' \in
\cD_j(\R^{n-1})}[\ell(Q')]^{-sp-\frac{np}{2}+1} |t_{Q' \times
[0,\ell(Q'))}|^p\int_{Q'}[\wz\omega(x',2^{-j})]^{-p}\,dx'
\right]^{\frac qp}\right\}^\frac1q\\
&&\hs\ls\left\{ \sum_{j \in \Z} \left[ \sum_{Q' \in
\cD_j(\R^{n-1})}[\ell(Q')]^{-sp-\frac{np}{2}} |t_{Q' \times
[0,\ell(Q'))}|^p\int_{Q}[\omega(x,2^{-j})]^{-p}\,dx
\right]^{\frac qp}\right\}^\frac1q\\
&&\hs\ls \|t\|_{\dsbh},
\end{eqnarray*}
which implies that ${\rm Tr}$ is well defined and bounded from
$\dbh$ to $B\dot{H}_{p,q}^{s-\frac1p,\frac{n}{n-1}\tau}(\R^{n-1})$.

Let us show that ${\rm Tr}$ is surjective. To this end, for any $f
\in B\dot{H}_{p,q}^{s-\frac1p,\frac{n}{n-1}\tau}(\R^{n-1})$, by
Theorem \ref{t3.3}, there exist smooth atoms $\{a_{Q'}\}_{Q' \in
\cD(\R^{n-1})}$ for
$B\dot{H}_{p,q}^{s-\frac1p,\frac{n}{n-1}\tau}(\R^{n-1})$ and
coefficients $t\equiv\{t_{Q'}\}_{Q' \in \cD(\R^{n-1})}$ such that
$f=\sum_{Q' \in \cD(\R^{n-1})}t_{Q'}a_{Q'}$ in $\cS_\fz'(\R^{n-1})$
and
$\|t\|_{b\dot{H}_{p,q}^{s-\frac{1}{p},\frac{n}{n-1}\tau}(\R^{n-1})}
\ls\|f\|_{B\dot{H}_{p,q}^{s-\frac1p,\frac{n}{n-1}\tau}(\R^{n-1})}.$
Let $\varphi \in C^\infty_c(\R)$ with $\supp \varphi
\subset(-\frac12,\frac12)$ and $\varphi(0)=1$. For all $Q' \in
\cD(\R^{n-1})$ and $x\in\R$, set
$\varphi_{Q'}(x)\equiv\varphi(2^{-\log_2 \ell(Q')}x)$. Under this
notation, we define $F\equiv\sum_{Q' \in \cD(\R^{n-1})} t_{Q'}a_{Q'}
\otimes \varphi_{Q'}.$ It is easy to check that for all $Q' \in
\cD(\R^{n-1})$, $[\ell(Q')]^{-\frac12}a_{Q'} \otimes \varphi_{Q'}$
is a smooth atom for $\dbh$ supported near $Q'\times [0, \ell(Q'))$.
Hence, to show $F\in\dbh$, by Theorem \ref{t3.3}, it suffices to
prove that
\begin{equation*}
\l\|\{[\ell(Q')]^\frac12t_{Q'}\}_{Q'\in{\mathcal D}
(\R^{n-1})}\r\|_{\dsbh}
\ls\|f\|_{B\dot{H}_{p,q}^{s-\frac1p,\frac{n}{n-1}\tau}(\R^{n-1})}.
\end{equation*}

Let $\wz\omega$ satisfy $\int_{\R^{n-1}}[N\wz\omega(x')]^{(p\vee
q)'} d\wz{H}^{n\tau(p\vee q)'}(x')\le1$ and
$$\left\{ \sum_{j \in \Z} \left[ \sum_{Q' \in
\cD_j(\R^{n-1})}|Q'|^{-(\frac{s-1/p}{n-1}+\frac12)p}|t_{Q'}|^p\int_{Q'}
[\wz\omega(x',2^{-j})]^{-p}\,dx' \right]^{\frac
qp}\right\}^\frac1q\ls
\|t\|_{b\dot{H}_{p,q}^{s-\frac{1}{p},\frac{n}{n-1}\tau}(\R^{n-1})}.$$
By Lemma \ref{l4.1}, for each $\lambda\in(0,\fz)$, there exists a
ball covering $\{B^*_m\}_m\equiv \{B(x_{B^*_m},r_{B^*_m})\}_m$ of
$\wz{E}_\lambda\equiv\{x'\in\R^{n-1}:\,[N\wz\omega(x')]^{(p\vee q)'}
>\lambda\}$ such that $\sum_m r_{B^*_m}^{n\tau(p\vee
q)'}\sim\wz{H}_*^{n\tau(p\vee
q)'}(\wz{E}_\lambda)\sim\wz{H}^{n\tau(p\vee q)'}(\wz{E}_\lambda)$
and that $r_{B_m^*}>{\rm
dist}(x_{B^*_m},\partial\wz{E}_\lambda)/10000$ for all $m$. For all
$x=(x',x_n)\in\rn$ and $t\in(0,\fz)$, define $\omega(x,t)\equiv
\wz\omega(x',t)\chi_{[0,t)}(x_n)$. Notice that if
$N\omega(x',x_n)>\lambda^{\frac1{(p\vee q)'}}$, then
$\wz\omega(y',t)=\omega((y',y_n),t)>\lambda^{\frac1{(p\vee q)'}}$
for some $|(y',y_n)-(x',x_n)|<t$ and $y_n\in[0,t)$. Then
$N\wz\omega(y')>\lambda^{\frac1{(p\vee q)'}}$ and thus, $y'\in
B^*_m$ for some $m$. Since for all $z'\in B(y',t)$,
$N\wz\omega(z')\ge\wz\omega(y',t)>\lambda^{\frac1{(p\vee q)'}}$, we
see that $B(y',t)\subset \wz{E}_\lambda\subset (\cup_m B^*_m)$, and
hence, $t\le 10000 r_{B^*_m}$. Notice that $x_n\in[0,t)$. We have
$(x',x_n)\in (20000 B^*_m)\times [0,20000r_{B^*_m})$ and
$E_\lambda\subset \cup_{m} (20000 B^*_m)\times [0,20000r_{B^*_m})$,
which further implies that $H^{n\tau(p\vee q)'}(E_\lambda)\ls \sum_m
r_{B^*_m}^{n\tau(p\vee q)'}\ls \wz{H}^{n\tau(p\vee
q)'}(\wz{E}_\lambda)$ and
\begin{eqnarray*}
\int_\rn [N\omega(x',x_n)]^{(p\vee q)'}\,dH^{n\tau(p\vee q)'}(x)
&&=\int_0^\fz H^{n\tau(p\vee q)'}(E_\lambda)\,d\lambda\ls \int_0^\fz
\wz{H}^{n\tau(p\vee q)'}(\wz{E}_\lambda)\,d\lambda\\
&&\ls \int_{\rr^{n-1}}[N\wz{\omega}(x')]^{(p\vee
q)'}\,d\wz{H}^{n\tau(p\vee q)'}(x')\ls1.
\end{eqnarray*}
Therefore, we have
\begin{eqnarray*}
&&\l\|\{[\ell(Q')]^\frac12t_{Q'}\}_{Q'\in{\mathcal D}
(\R^{n-1})}\r\|_{\dsbh}\\
&&\hs\ls\left\{ \sum_{j \in \Z} \left[ \sum_{Q' \in
\cD_j(\R^{n-1})}[\ell(Q')]^{-(\frac{s}{n}+\frac12)pn+\frac
p2}|t_{Q'}|^p \int_{Q'\times[0,\ell(Q'))}[\omega(x,2^{-j})]^{-p}\,dx
\right]^{\frac qp}\right\}^\frac1q\\
&&\hs\ls\left\{ \sum_{j \in \Z} \left[ \sum_{Q' \in
\cD_j(\R^{n-1})}|Q'|^{-(\frac{s-1/p}{n-1}+\frac12)p}|t_{Q'}|^p
\int_{Q'}[\wz\omega(x',2^{-j})]^{-p}\,dx' \right]^{\frac
qp}\right\}^\frac1q\\
&&\hs\ls
\|t\|_{b\dot{H}_{p,q}^{s-\frac{1}{p},\frac{n}{n-1}\tau}(\R^{n-1})}\ls\|
f \|_{B\dot{H}_{p,q}^{s-\frac1p,\frac{n}{n-1}\tau}(\R^{n-1})},
\end{eqnarray*}
which implies that $F\in\dbh$ and $\|F\|_{\dbh}\ls \| f
\|_{B\dot{H}_{p,q}^{s-\frac1p,\frac{n}{n-1}\tau}(\R^{n-1})}$.
Furthermore, the definition of $F$ implies ${\rm Tr}(F)=f$, which
completes the proof of Theorem \ref{t4.2}.
\end{proof}

We point out that Theorem \ref{t4.2} generalizes the corresponding
classical results on Besov and Triebel-Lizorkin spaces for $p\in
(1,\fz)$ and $q\in[1,\fz)$ by taking $\tau=0$; see, for example,
Jawerth \cite[Theorem~5.1]{j77}, \cite[Theorem~2.1]{j78} and
Frazier-Jawerth \cite[Theorem~11.1]{fj}.

\medskip

\noindent{\bf Acknowledgments}

\medskip

The authors would like to thank the referees to point out us
some references.

\bigskip

\noindent Wen Yuan:

\smallskip

\noindent School of Mathematical Sciences, Beijing Normal
University, Laboratory of Mathematics and Complex Systems, Ministry
of Education, Beijing 100875, People's Republic of China

\smallskip

\noindent{\it E-mail:} \texttt{wyuan@mail.bnu.edu.cn}

\bigskip

\noindent Yosihiro Sawano:

\smallskip

\noindent Department of Mathematics, Kyoto University,
Kitashirakawa Oiwakecho, Kyoto 606-8502, Japan

\smallskip

\noindent{\it E-mail}: \texttt{yosihiro@math.kyoto-u.ac.jp}

\bigskip

\noindent Dachun Yang (Corresponding author):

\smallskip

\noindent School of Mathematical Sciences, Beijing Normal
University, Laboratory of Mathematics and Complex Systems, Ministry
of Education, Beijing 100875, People's Republic of China

\smallskip

\noindent{\it E-mail:} \texttt{dcyang@bnu.edu.cn}


\begin{thebibliography}{99}

\bibitem{ad} D.~R.~Adams, A note on Choquet integrals
with respect to Hausdorff capacity, in: Function spaces and
applications (Lund, 1986), 115-124, Lecture Notes in Math. 1302,
Springer, Berlin, 1988.

\vspace{-0.3cm}
\bibitem{ao} T.~Aoki, Locally bounded linear topological space,
Proc.~Imp.~Acad.~Tokyo 18 (1942), 588-594.

\vspace{-0.3cm}
\bibitem{b05} M.~Bownik, Atomic and molecular decompositions
of anisotropic Besov spaces, Math.~Z. 250 (2005), 539-571.

\vspace{-0.3cm}
\bibitem{b07} M.~Bownik, Anisotropic Triebel-Lizorkin spaces with doubling
measures, J.~Geom. Anal. 17 (2007), 387-424.

\vspace{-0.3cm}
\bibitem{bh}
M.~Bownik and K.~Ho, Atomic and molecular decompositions of
anisotropic Triebel-Lizorkin spaces,
Trans.~Amer.~Math.~Soc. 358 (2006), 1469-1510.

\vspace{-0.3cm}
\bibitem{dx} G.~Dafni and J.~Xiao, Some new tent spaces and
duality theorems for fractional Carleson measures
and $Q_\alpha({\mathbb R}^n)$, J.~Funct.~Anal.
208 (2004), 377-422.

\vspace{-0.3cm}
\bibitem{du} J.~Duoandikoetxea, Fourier Analysis,
Graduate Studies in Mathematics 29, American Mathematical Society,
Providence, R. I., 2001.

\vspace{-0.3cm}
\bibitem{ejpx} M.~Ess\'en, S.~Janson, L.~Peng and J.~Xiao,
$Q$ spaces of several real variables, Indiana Univ.~Math.~J. 49
(2000), 575-615.

\vspace{-0.3cm}
\bibitem{fhjw} M.~Frazier, Y.-S.~Han, B.~Jawerth and G.~Weiss,
The $T1$ theorem for Triebel-Lizorkin spaces,
Harmonic analysis and partial differential
equations (El Escorial, 1987), 168-181,
Lecture Notes in Math., 1384, Springer, Berlin, 1989.

\vspace{-0.3cm}
\bibitem{fj85} M.~Frazier and B.~Jawerth, Decomposition of Besov spaces,
Indiana Univ. Math. J. 34 (1985), 777-799.

\vspace{-0.3cm}
\bibitem{fj88} M.~Frazier and B.~Jawerth, The $\phi$-transform and
applications to distribution spaces, Function spaces and
applications (Lund, 1986), 223--246, Lecture Notes in Math.,
1302, Springer, Berlin, 1988.

\vspace{-0.3cm}
\bibitem{fj} M.~Frazier and B.~Jawerth,
A discrete transform and decompositions of distribution spaces,
J.~Funct.~Anal. 93 (1990), 34-170.

\vspace{-0.3cm}
\bibitem{fjw} M.~Frazier, B.~Jawerth and G.~Weiss, Littlewood-Paley Theory and
The Study of Function Spaces, CBMS Regional Conference Series in
Mathematics, 79. Published for the Conference Board of the
Mathematical Sciences, Washington, D. C.; by the American
Mathematical Society, Providence, R. I., 1991.

\vspace{-0.3cm}
\bibitem{ftw} M.~Frazier, R.~Torres and G.~Weiss,
The boundedness of Calder\'on-Zygmund operators
on the spaces $\dot F^{\alpha,q}_p$, Rev.~Mat.~Iberoamericana
4 (1988), 41-72.

\vspace{-0.3cm}
\bibitem{g08} L. Grafakos,  Modern Fourier Analysis,
Second Edition, Graduate Texts in Math., No. 250, Springer,
New York, 2008.

\vspace{-0.3cm}
\bibitem{gt}
L.~Grafakos and R.~H.~Torres, Pseudodifferential operators with
homogeneous symbols, Michigan Math.~J. 46 (1999), 261-269.

\vspace{-0.3cm}
\bibitem{j77} B.~Jawerth, Some observations on Besov
and Lizorkin-Triebel spaces, Math.~Scand. 40 (1977), 94-104.

\vspace{-0.3cm}
\bibitem{j78} B.~Jawerth, The trace of Sobolev and
Besov spaces if $0<p<1$, Studia Math. 62 (1978), 65-71.

\vspace{-0.3cm}
\bibitem{m92} Y.~Meyer, Wavelets and Operators,
Translated from the 1990 French original by D. H. Salinger,
Cambridge Studies in Advanced Mathematics, 37,
Cambridge University Press, Cambridge, 1992.

\vspace{-0.3cm}
\bibitem{mc97} Y.~Meyer and R.~R.~Coifman, Wavelets. Calder\'on-Zygmund
and Multilinear Operators, Translated from the 1990 and 1991
French originals by David Salinger, Cambridge Studies in
Advanced Mathematics, 48, Cambridge University Press, Cambridge, 1997.

\vspace{-0.3cm}
\bibitem{st}
Y.~Sawano and H.~Tanaka,
Decompositions of Besov-Morrey spaces and
Triebel-Lizorkin-Morrey spaces,
Math.~Z. 257 (2007), 871-905.

\vspace{-0.3cm}
\bibitem{syy1}
Y.~Sawano, D.~Yang and W.~Yuan, New Applications of Besov-Type and
Triebel-Lizorkin Type Spaces, J.~Math.~Anal.~Appl. 363 (2010), 73-85.

\vspace{-0.3cm}
\bibitem{t90} R. H. Torres, Continuity properties of pseudodifferential
operators of type $1,1$, Comm.~Partial~Differential~Equations 15
(1990), 1313-1328.

\vspace{-0.3cm}
\bibitem{t91} R. H. Torres, Boundedness results for operators
with singular kernels on distribution spaces, Mem.~Amer.~Math.~Soc.
90 (1991), no. 442, viii+172 pp.

\vspace{-0.3cm}
\bibitem{t83} H.~Triebel, Theory of Function Spaces,
Birkh\"auser Verlag, Basel, 1983.

\vspace{-0.3cm}
\bibitem{t95} H.~Triebel, Interpolation Theory, Function Spaces,
Differential Operators, Second edition, Johann Ambrosius Barth, Heidelberg, 1995.

\vspace{-0.3cm}
\bibitem{X1} J.~Xiao, Holomorphic $Q$ Classes, Lecture Notes in
Math. 1767, Springer, Berlin, 2001.

\vspace{-0.3cm}
\bibitem{X2} J.~Xiao, Geometric $Q_p$ Functions,
Birkh\"auser Verlag, Basel, 2006.

\vspace{-0.3cm}
\bibitem{yy0}
D.~Yang and W.~Yuan, A note on dyadic Hausdorff capacities,
Bull.~Sci.~Math. 132 (2008), 500-509.

\vspace{-0.3cm}
\bibitem{yy1}
D.~Yang and W.~Yuan, A new class of function spaces connecting
Triebel-Lizorkin spaces and $Q$ spaces, J.~Funct.~Anal. 255 (2008),
2760-2809.

\vspace{-0.3cm}
\bibitem{yy2}
D.~Yang and W.~Yuan,
New Besov-type spaces and Triebel-Lizorkin-type
spaces including $Q$ spaces,
Math.~Z. (to appear).
\end{thebibliography}
\end{document}